\documentclass[reqno]{amsart}
\usepackage{todonotes}

\usepackage{amssymb}
\usepackage{amsmath,texdraw}
\usepackage{mathtools}
\usepackage{amsthm}
\usepackage{fullpage}
\usepackage{enumitem}
\usepackage{xcolor}
\usepackage{graphicx}
\usepackage{pgfarrows,pgfnodes}
\usepackage{url}
\usepackage{textcomp}
\usepackage{ytableau}
\usepackage{dsfont}
\usepackage{tikz}
\usepackage[toc,page]{appendix}
\usepackage{tcolorbox}
\definecolor{foge}{rgb}{0.1, 0.6, 0.1}

\numberwithin{equation}{section}
\numberwithin{figure}{section}
\newtheorem{theorem}{Theorem}[section]
\newtheorem{cor}[theorem]{Corollary}
\newtheorem{fig}[figure]{Figure}
\newtheorem{lem}[theorem]{Lemma}
\newtheorem{prop}[theorem]{Proposition}
\newtheorem{rem}[theorem]{Remark}
\newtheorem{ex}[theorem]{Example}
\newtheorem{exs}[theorem]{Examples}
\theoremstyle{definition} \newtheorem{defn}[theorem]{Definition}
\newcommand{\m}{\medbreak}
\newcommand{\bi}{\bigbreak}

\DeclareMathOperator{\inter}{int}

\newcommand{\Thm}[1]{Theorem \ref{#1}}

\theoremstyle{remark}
\newtheorem*{remark}{Remark}

\theoremstyle{definition}

\newtheorem{definition}[theorem]{Definition}

\newcommand{\g}{\mathfrak g}
\newcommand{\h}{\mathfrak h}
\newcommand{\sv}{\sigma^{\vee}}

\newcommand{\B}{\mathcal{B}}
\newcommand{\Bb}{\mathbb{B}}
\newcommand{\C}{\mathcal{C}}
\newcommand{\vv}{v^\vee}
\newcommand{\I}{\{0, \dots , n-1\}}
\newcommand{\Pp}{\mathcal{P}}
\newcommand{\Ppp}{\Pp^{\gg}_{\co}}
\newcommand{\Pppp}{\Pp^{\gtrdot}_{\co}}
\newcommand{\co}{c_{g}}

 \newcommand{\iso} {\buildrel \sim \over \rightarrow}

\newcommand{\ot}{\otimes}
\newcommand{\ov}{\overline}
\newcommand{\p}{\mathfrak p}

\newcommand{\N}{\mathbb{N}}

\newcommand{\Z}{\mathbb {Z}}
\newcommand{\wt}{\overline{\rm wt}}

\newcommand{\Ll}{\lambda}

\newcommand{\ssss}{\{0,\ldots,s-1\}}

\title[Perfect $A_{n-1}^{(1)}$ crystals and characters]{Generalisations of Capparelli's and Primc's identities, II: perfect $A_{n-1}^{(1)}$ crystals and explicit character formulas}
\author{Jehanne Dousse}
\address{Univ Lyon, CNRS, Universit\'e Claude Bernard Lyon 1, UMR5208, Institut Camille Jordan, F-69622 Villeurbanne, France}
\email{dousse@math.cnrs.fr}

\author{Isaac Konan}
\address{IRIF, Universit\'e de Paris, Bâtiment Sophie Germain, Case courrier 7014, 8 Place Aurélie Nemours, 75205 Paris Cedex 13, France}
\email{konan@irif.fr}

\begin{document}

\begin{abstract}
In the first paper of this series, we gave infinite families of coloured partition identities which generalise Primc's and Capparelli's classical identities. 

In this second paper, we study the representation theoretic consequences of our combinatorial results.
First, we show that the difference conditions we defined in our $n^2$-coloured generalisation of Primc's identity, which have a very simple expression, are actually the energy function with values in $\{0,1,2\}$ for the perfect crystal of the tensor product of the vector representation and its dual in $A_{n-1}^{(1)}$.

Then we introduce a new type of partitions, grounded partitions, which allows us to retrieve connections between character formulas and partition generating functions without having to perform a specialisation.

Finally, using the formulas for the generating functions of our generalised partitions, we recover the Kac-Peterson character formula for the characters of all the irreducible highest weight $A_{n-1}^{(1)}$-modules of level $1$, and give a new character formula as a sum of infinite products with obviously positive coefficients in the generators $e^{- \alpha_i} \ (i \in \{0, \dots , n-1\}),$ where the $\alpha_i$'s are the simple roots.
\end{abstract}

\maketitle

\section{Introduction and statement of results}
\subsection{Background}
A \emph{partition} $\lambda$ of a positive integer $n$ is a non-increasing sequence of natural numbers $(\lambda_1,\dots,\lambda_s)$ whose sum is $n$, written as the sum $\lambda_1+\cdots +\lambda_s$. The numbers $\lambda_1,\dots,\lambda_s$ are called the \emph{parts} of $\lambda$, and $|\lambda|=n$ is the \emph{weight} of $\lambda$. For example, the partitions of $4$ are $4, 3+1, 2+2, 2+1+1,$ and $ 1+1+1+1.$

The Rogers-Ramanujan identities \cite{RogersRamanujan} state that for $a=0$ or $1$, the number of partitions of $n$ such that the difference between two consecutive parts is at least $2$ and the part $1$ appears at most $1-a$ times is equal to the number of partitions of $n$ into parts congruent to $\pm (1+a) \mod 5.$
In the 1980's, Lepowsky and Wilson \cite{Lepowsky,Lepowsky2} gave an interpretation and proof of these identities in terms of characters for level $3$ standard modules of the affine Lie algebra $A_1^{(1)}$ by using vertex operators.
Since then, a very fruitful interaction between partition identities and representation theory has been developed, see for example \cite{Capparelli,Meurman,Meurman2,Meurman3,Nandi,Primc1,PrimcSikic,Siladic}. More detail on the history of this field can be found in the first paper of this series \cite{DK19}.

\medskip
In the present paper, we focus on the interaction between partition identities and crystal base theory. Crystal bases were introduced independently by Kashiwara \cite{Kas90} and Lusztig \cite{Lusz90} to study representations of quantum algebras, which are $q$-deformations of universal enveloping algebras of classical Lie algebras. They have a nice combinatorial structure, and admit particularly simple tensor products.

One of the most important questions in representation theory is finding nice explicit formulas for characters of representations.
If $\hat{\mathfrak{g}}$ is an affine Lie algebra, and $V$ an irreducible module of $\hat{\mathfrak{g}}$ with highest weight $\Lambda$, then by definition, the character $\mathrm{ch}(V)$ of $V$ multiplied by $e^{-\Lambda}$ can be expressed as a power series in $e^{-\alpha_0}, \dots, e^{-\alpha_{n-1}}$ with positive coefficients, where $\alpha_0, \dots , \alpha_{n-1}$ are the simple roots of $\hat{\mathfrak{g}}$. However, finding explicit expressions for characters is not easy. The most famous example, the \textit{Weyl-Kac character formula} \cite{Kac}, gives a beautiful factorized expression for the character, but the coefficients of the monomials in $e^{-\alpha_i}$ in this expression are not obviously positive.

Kang, Kashiwara, Misra, Miwa, Nakashima, and Nakayashiki \cite{KMN2a,KMN2b} introduced the theory of perfect crystals to find such nice expressions for characters via the so-called \textit{(KMN)$^2$ crystal base character formula}.
It allows one to construct crystals of irreducible highest weight modules for all classical weights of the same level. 
Then the crystal base character formula allows one to identify these perfect crystals with partitions satisfying certain difference conditions, which in certain cases gives rise to character formulas as partition generating functions. However, these formulas are in general obtained after doing a specialisation, for example replacing all the $e^{-\alpha_i}$'s by $q$ (which is the principal specialisation). In this paper, we will prove a non-specialised character formula, with obviously positive coefficients, for all the irreducible highest weight $A_{n-1}^{(1)}$-modules of level $1$.

\medskip
But first, let us present our starting point, Primc's partition identity (again, more detail can be found in our first paper \cite{DK19}). 
In \cite{Primc}, Primc used the (KMN)$^2$ crystal base character formula to study level $1$ standard modules of $A_1^{(1)}$ and  $A_2^{(1)}$. He computed an energy function for the perfect crystal of the tensor product of the vector representation and its dual in $A_{1}^{(1)}$ and  $A_2^{(1)}$, and through the crystal base character formula, he gave
the principal specialisation of the character formula in terms of partitions with difference conditions.

In the $A_1^{(1)}$ case, the energy matrix of the perfect crystal associated to the tensor product of the vector representation and its dual is the following:
\begin{equation} \label{Primcmatrix4}
P_2=\bordermatrix{\text{} & a_1b_0 & a_0b_0 & a_1b_1 & a_0b_1 \cr a_1b_0 & 2&1&2&2 \cr a_0b_0 &1&0&1&1 \cr a_1b_1 &0&1&0&2 \cr a_0b_1&0&1&0&2},
\end{equation}
and in $A_2^{(1)}$, the energy matrix is given by:
\begin{equation} \label{Primcmatrix9}
P_3=\bordermatrix{\text{} & a_2b_0 & a_2b_1 & a_1b_0 & a_0b_0 & a_2b_2 & a_1b_1 & a_0b_1 & a_1b_2 & a_0b_2 \cr a_2b_0 & 2&2&2&1&2&2&2&2&2 \cr a_2b_1 &1&2&1&1&2&1&2&2&2 \cr a_1b_0 &1&1&2&1&1&2&2&2&2 \cr a_0b_0 & 1&1&1&0&1&1&1&1&1 \cr a_2b_2 &0&0&1&1&0&1&1&2&2 \cr a_1b_1 &0&1&0&1&1&0&2&1&2 \cr a_0b_1 &0&1&0&1&1&0&2&1&2 \cr a_1b_2 &0&0&1&1&0&1&1&2&2 \cr a_0b_2 &0&0&0&1&0&0&1&1&2}.
\end{equation}

Consider coloured partitions satisfying the difference conditions of \eqref{Primcmatrix4} (resp. \eqref{Primcmatrix9}), where the coefficient $(i,j)$ in the matrix gives the minimal difference between consecutive parts coloured $i$ and $j$. Using the Weyl-Kac charater formula \cite{Kac}, Primc proved that in both cases, when performing the principal specialisation (corresponding to some dilations on the variables in the generating function), the generating function for such partitions reduces to $\frac{1}{(q;q)_{\infty}}$, which is simply the generating function for partitions. Here we used, for $n \in \N \cup \{\infty\},$ the standard $q$-series notation 
$$(a;q)_n := (1-a)(1-aq)\cdots(1-aq^{n-1}).$$

In the first paper of this series \cite{DK19}, we gave a large family of coloured partition identities which generalise and refine Primc's identities. To do so, we gave difference conditions which generalise both \eqref{Primcmatrix4} and \eqref{Primcmatrix9}.
Let $(a_{n})_{n\in\N}$ and $(b_{n})_{n\in\N}$ be two sequences of colour symbols.  For all $i,k,i',k' \in \N$, we defined the minimal difference $\Delta$ in the following way:
\begin{equation}
\label{eq:Delta}
\Delta(a_i b_k, a_{i'} b_{k'}) = \chi(i \geq i') - \chi(i=k=i')+\chi(k \leq k') - \chi(k=i'=k'),
\end{equation}
where $\chi(prop)$ equals $1$ if the proposition $prop$ is true and $0$ otherwise.

Restricting $\Delta$ to colours $a_ib_j$ for $i,j \in \{0,1\}$ gives \eqref{Primcmatrix4}, and restricting it to colours $a_ib_j$ for $i,j \in \{0,1,2\}$ gives \eqref{Primcmatrix9}.

Our general theorem in \cite{DK19} gives the generating function for partitions
 $\lambda_1 + \cdots + \lambda_s$ into parts coloured $a_ib_j$ for all $i,j \in \{0, \dots, n-1\}$, satisfying the difference conditions 
$$\lambda_j - \lambda_{j+1} \geq \Delta(c(\lambda_j),c(\lambda_{j+1})),$$
where for all $j$, $c(\lambda_j)$ denotes the colour of the part $\lambda_j$. Such partitions are called generalised Primc partitions and their set is denoted by $\mathcal{P}_n$.
Let $P_{n}(m;u_0,\dots,u_{n-1};v_0,\dots,v_{n-1})$ be the number of partitions of $m$ in $\mathcal{P}_n$, such that for $i \in \{0, \dots, n-1\},$ the symbol $a_i$ (resp. $b_i$) appears $u_i$ (resp. $v_i$) times in the colour sequence.
Defining the generating function
$$G^P_{n}(q;b_0,\cdots,b_{n-1})
:=\sum_{m,u_0,\dots,u_{n-1},v_0,\dots,v_{n-1} \geq 0} P_{n}(n;u_0,\dots,u_{n-1};v_0,\dots,v_{n-1})q^m b_0^{v_0-u_0} \cdots b_{n-1}^{v_{n-1}-u_{n-1}},$$
we showed the following.

\begin{theorem}\cite[Theorem 1.27]{DK19}
\label{theorem:main2}
Let $n$ be a positive integer. We have:
\begin{align}
G^P_{n}(q;b_0,\cdots,b_{n-1})&= [x^0]\prod_{i=0}^{n-1}(-b^{-1}_ix q;q)_{\infty}(-b_ix^{-1};q)_{\infty}\nonumber
\\&=\frac{1}{(q;q)_{\infty}^{n}}\sum_{\substack{s_1, \dots, s_{n-1}\in \Z\\s_n=0}} b_0^{s_1}\prod_{i=1}^{n-1} b_i^{-s_i+s_{i+1}}q^{s_i(s_i-s_{i+1})} \label{eq:jacob}
\\&=\frac{1}{(q;q)_{\infty}}\left(\prod_{i=1}^{n-1}\frac{\left(q^{i(i+1)};q^{i(i+1)}\right)_{\infty}}{(q;q)_{\infty}}\right)  \sum_{\substack{r_1, \dots, r_{n-1}:\\0 \leq r_j \leq j-1\\r_n=0}} \prod_{i=1}^{n-1} b_i^{-r_i+r_{i+1}}q^{r_i(r_i-r_{i+1})}\nonumber
\\& \qquad \qquad \qquad \times \left(- \left(\prod_{\ell=0}^{i-1} b_{\ell}b_i^{-1} \right) q^{\frac{i(i+1)}{2}+(i+1)r_i-ir_{i+1}};q^{i(i+1)}\right)_{\infty} \label{eq:formulefinale}
\\& \qquad \qquad \qquad \times \left(- \left(\prod_{\ell=0}^{i-1} b_{i}b_{\ell}^{-1} \right) q^{\frac{i(i+1)}{2}-(i+1)r_i+ir_{i+1}};q^{i(i+1)}\right)_{\infty}\,\cdot\nonumber
\end{align}
\end{theorem}
We can obtain a product formula for our generating function by doing the following dilations, which correspond to the principal specialisation that Primc considered in his paper:
\begin{equation}\label{eq:dilation}
 \begin{cases}
  q &\mapsto \ \ q^n\\
  b_i&\mapsto \ \ q^i \quad \text{for all }i\in\{0,\ldots,n-1\}\,\cdot
 \end{cases}
\end{equation}
\begin{cor}\cite[Corollary 1.26]{DK19}
 By doing the transformations described in \eqref{eq:dilation}, we obtain the generating function for classical integer partitions:
 \begin{align*}
 G^P_{n}(q^n;1,\cdots,q^{n-1})&= 
  [x^0]\prod_{i=0}^{n-1}(-q^{n-i}x;q^n)_{\infty}(-q^ix^{-1};q^n)_{\infty} \\
  &= [x^0](-qx;q)_{\infty}(-x^{-1};q)_{\infty} \\
  &= \frac{1}{(q;q)_\infty}\,\cdot
 \end{align*}
\end{cor}
The cases $n=2$ and $n=3$ in the corollary above recover Primc's original results.

\medskip
In \cite{DK19}, we also gave a multi-parameter family of generalisations of Capparelli's identity \cite{Capparelli}, another partition identity which originally arose from representation theory via the theory of vertex operators. Let us also state this generalisation, as it gives a different (but related) expression for the character formula.

\begin{definition}
Let $\pi=\pi_1+ \dots +\pi_r$ be a partition. We say that another partition $\lambda=\lambda_1+\cdots+\lambda_s$ \textit{contains the pattern} $\pi$ if there is some index $i$ such that
$$\lambda_i=\pi_1,\quad  \lambda_{i+1} = \pi_2, \quad \dots , \quad \lambda_{i+r-1}=\pi_r.$$
If $\lambda$ does not contain the pattern $\pi$, we say that $\lambda$ \textit{avoids} $\pi$.
\end{definition}

Let us recall from \cite{DK19} some conditions that the parameters in our generalisation need to satisfy.

\begin{definition}
\label{def:cond1}
A function $\delta$ is said to satisfy Condition 1 if it is defined on the set of colours $\{a_k b_\ell : k \neq \ell\}$, has integer values, and for all $k,\ell$,
$$\min\{k,\ell\}<\delta(a_kb_\ell)\leq \max\{k,\ell\}.$$
\end{definition}

\begin{definition}
\label{def:cond2}
A function $\gamma$ is said to satisfy Condition 2 if it is defined on the set of pairs of colours $\{(a_{k_1} b_{\ell_1}, a_{k_2} b_{\ell_2}): k_1 \neq \ell_1, k_2 \neq \ell_2 \}$, has integer values, and if for all $k_1,k_2,\ell_1,\ell_2$, it satisfies the following:
\begin{itemize}
\item If $\max\{k_1,\ell_2\}<\min\{k_2,\ell_1\}$, we have 
$$\max\{k_1,\ell_2\} <\gamma(a_{k_1} b_{\ell_1}, a_{k_2} b_{\ell_2})\leq  \min\{k_2,\ell_1\}.$$
\item If $k_1>\ell_1$, $k_2>\ell_2$, and $\{\ell_2+1,\ldots,k_2\}\setminus \{\ell_1+1,\ldots,k_1\}\neq \emptyset$, we have
$$\gamma(a_{k_1} b_{\ell_1}, a_{k_2} b_{\ell_2})\in \{\ell_2+1,\ldots,k_2\}\setminus \{\ell_1+1,\ldots,k_1\}.$$
\item If $k_1<\ell_1$, $k_2<\ell_2$, and $\{k_1+1,\ldots,\ell_1\}\setminus\{k_2+1,\ldots,\ell_2\}\neq \emptyset$, we have 
$$\gamma(a_{k_1} b_{\ell_1}, a_{k_2} b_{\ell_2})\in \{k_1+1,\ldots,\ell_1\}\setminus\{k_2+1,\ldots,\ell_2\}.$$
\end{itemize}
\end{definition}

These functions now allow us to define forbidden patterns and generalised Capparelli partitions.

\begin{definition}
\label{def:Capp_conditions}
Let $n$ be a positive integer, and let $\delta$ and $\gamma$ be functions satisfying Conditions 1 and 2, respectively. We define $\mathcal{C}_n(\delta,\gamma)$, the set of generalised Capparelli partitions related to $\delta$ and $\gamma$, to be the set of partitions $\lambda$ such that
\begin{itemize}
\item $\lambda \in \mathcal P_n$,
\item $\lambda$ has no part coloured $a_0b_0$,
\item $\lambda$ does not contain any of the following patterns, where $p$ is any positive integer:
\begin{itemize}
\item for any $i \in \{1, \dots , n-1\},$
$$p_{a_ib_i}+p_{a_ib_i},$$
(i.e. free colours cannot repeat)
\item for any $k_1,k_2,\ell_1,\ell_2$ such that $\max\{k_1,\ell_2\}<\min\{k_2,\ell_1\}$ and $i=\gamma(a_{k_1}b_{\ell_1},a_{k_2}b_{\ell_2})$,
$$p_{a_{k_1}b_{\ell_1}}+p_{a_ib_i}+p_{a_{k_2}b_{\ell_2}},$$
\item for any $k_2>\ell_2$,
\begin{itemize}
\item for any $2\leq u\leq \infty$, any $k_1, \ell_1$, and $i=\delta(a_{k_2}b_{\ell_2})$,
$$(p+u)_{a_{k_1}b_{\ell_1}}+p_{a_ib_i}+p_{a_{k_2}b_{\ell_2}},$$
(here we take the convention that $u=\infty$ if the pattern $p_{a_ib_i}+p_{a_{k_2}b_{\ell_2}}$ is at the beginning of the partition)
\item for any $k_1\leq \ell_1$, and $i=\delta(a_{k_2}b_{\ell_2})$,
$$(p+1)_{a_{k_1}b_{\ell_1}}+p_{a_ib_i}+p_{a_{k_2}b_{\ell_2}},$$
\item for any $k_1>\ell_1$ such that $\{\ell_2+1,\ldots,k_2\}\setminus \{\ell_1+1,\ldots,k_1\}\neq \emptyset$, and for $i=\gamma(a_{k_1}b_{\ell_1},a_{k_2}b_{\ell_2})$,
$$(p+1)_{a_{k_1}b_{\ell_1}}+p_{a_ib_i}+p_{a_{k_2}b_{\ell_2}},$$
\end{itemize} 
\item for any $k_1<\ell_1$,
\begin{itemize}
\item for any $2\leq u\leq \infty$, any $k_2, \ell_2$, and $i=\delta(a_{k_1}b_{\ell_1})$,
$$p_{a_{k_1}b_{\ell_1}}+p_{a_ib_i}+(p-u)_{a_{k_2}b_{\ell_2}},$$
(here we take the convention that $u=\infty$ if the pattern $p_{a_{k_1}b_{\ell_1}}+p_{a_ib_i}$ is at the end of the partition)
\item for any $k_2\geq \ell_2$, and $i=\delta(a_{k_1}b_{\ell_1})$,
$$(p+1)_{a_{k_1}b_{\ell_1}}+(p+1)_{a_ib_i}+p_{a_{k_2}b_{\ell_2}},$$
\item for any $k_2<l_2$ such that $\{k_1+1,\ldots,\ell_1\}\setminus\{k_2+1,\ldots,\ell_2\}\neq \emptyset$, and for $i=\gamma(a_{k_1}b_{\ell_1},a_{k_2}b_{\ell_2})$,
$$(p+1)_{a_{k_1}b_{\ell_1}}+(p+1)_{a_ib_i}+p_{a_{k_2}b_{\ell_2}}.$$
\end{itemize} 
\end{itemize}
\end{itemize}
\end{definition}

Let $C_{n}(m;\delta,\gamma;u_0,\dots,u_{n-1};v_0,\dots,v_{n-1})$ be the number of  partitions of $m$ in $\mathcal{C}_n(\gamma,\delta)$, such that for $i \in \{0, \dots, n-1\},$ the symbol $a_i$ (resp. $b_i$) appears $u_i$ (resp. $v_i$) times in the colour sequence.
We define the generating function
$$G^C_{n}(\delta,\gamma;q;b_0,\cdots,b_{n-1})
:=\sum_{m,u_0,\dots,u_{n-1},v_0,\dots,v_{n-1} \geq 0} C_{n}(m;\delta,\gamma;u_0,\dots,u_{n-1};v_0,\dots,v_{n-1})q^m b_0^{v_0-u_0} \cdots b_{n-1}^{v_{n-1}-u_{n-1}}.$$
Through a bijection, we showed the following relation between our generalisations of Capparelli's and Primc's partitions with difference conditions.

\begin{theorem}\cite[Theorem 1.24]{DK19}
For all positive integers $n$ and all functions $\delta$ and $\gamma$  satisfying Conditions 1 and 2 respectively, we have
$$G^C_{n}(\delta,\gamma;q;b_0,\cdots,b_{n-1})= (q;q)_{\infty} G^P_{n}(q;b_0,\cdots,b_{n-1}).$$
\end{theorem}
Through Theorem \ref{theorem:main2}, the generating function for generalised Capparelli partitions can also be written as a sum of infinite products.

In this paper, we use our results above to give new character formulas.

\subsection{Statement of Results}
We will define all the necessary notions from crystal base theory in Section \ref{sec:reptheory}. For now, let us define a few notations which will allow us to state our main theorems.

Let $n$ be a positive integer, and consider the Cartan datum for the generalised Cartan matrix of affine type $A_{n-1}^{(1)}$.
We denote by $\bar P = \Z\Lambda_0 \oplus \cdots \oplus \Z\Lambda_{n-1}$ the lattice of the classical weights, where the elements $\Lambda_{\ell}$ $(\ell \in \{0, \dots , n-1\})$ are the 
fundamental weights. The set of all the level $1$ classical weights is given by $\bar P^+_1 = \{\Lambda_{\ell} : \ell\in \{0,\cdots, n-1\}\}$.
The null root is denoted by $\delta$, and the simple roots by $\alpha_i , \,i\in \{0,\cdots, n-1\}.$ Let $\B = \{v_i: i \in \{0,\cdots,n-1\}\}$ be the crystal of the vector representation  of $A_{n-1}^{(1)}$ and let $\B^\vee=\{\vv_i: i \in \{0,\cdots,n-1\}\}$ be its dual.
For all $v_i\in \B$, we denote by $\wt v_i\in \bar P$ the classical weight of $v_i$.
We finally set $\Bb$ to be the tensor product $\B\ot \B^\vee$.

\medskip
Given that \eqref{Primcmatrix4} and \eqref{Primcmatrix9} are energy matrices for perfect crystals coming from the tensor product of the vector representation and its dual in $A_1^{(1)}$ and $A_2^{(1)}$, respectively, it is natural to wonder whether our generalised difference conditions $\Delta$ defined in \eqref{eq:Delta} are also energy functions for certain perfect crystals. We answer this question in the affirmative by showing the following.
%
%The main relation between Primc generalised partitions and crystal bases theory consists in the fact that the energy function $H$ of
%$\Bb\ot \Bb$ indeed corresponds to the minimal difference of two consecutive coloured parts, as stated in the following theorem.  

\begin{theorem}\label{theorem:perfcrys}
Let $n$ be a positive integer, and let $\B$ denote the crystal of the vector representation  of $A_{n-1}^{(1)}$.
 The crystal $\Bb=\B\ot\B^\vee$ is a perfect crystal of level 1. 
 Furthermore, the energy function
 on $\Bb\ot\Bb$ such that $H((v_0\ot\vv_0)\ot(v_0\ot\vv_0))=0$
 satisfies for all $k, \ell , k', \ell' \in \{0, \dots, n-1\},$
 \begin{equation}\label{eq:equality}
  H((v_{\ell'}\ot\vv_{k'})\ot(v_{\ell}\ot\vv_{k})) = \Delta(a_kb_{\ell};a_{k'}b_{\ell'}),
 \end{equation}
where $\Delta$ is the minimal difference for generalised Primc partitions defined in \eqref{eq:Delta}. 
\end{theorem}
Primc showed Theorem \ref{theorem:perfcrys} in the cases $n=2$ and $n=3$. The theorem is still true when $n=1$, in which case the crystal $\B$ has a single vertex and a loop $0$, and the partitions corresponding to its crystal are simply the classical partitions. 

In \cite{Benkart}, Benkart, Frenkel, Kang, and Lee gave another formulation of the energy function of certain level $1$ perfect crystals of classical types, including the $A_{n-1}^{(1)}$-crystal studied in Theorem \ref{theorem:perfcrys}. However, they did not give a closed expression valid for all $k, \ell , k', \ell' \in \{0, \dots, n-1\}$ as we did in Theorem \ref{theorem:perfcrys} and \eqref{eq:Delta}. They used the fact that, when removing the $0$-arrows from the crystal graph on Figure \ref{fig:crystal} (see more details on crystals and energy functions in Sections \ref{sec:reptheory} and \ref{sec:An1}), the energy function $H$ is constant on each connected component, and gave a table with the value of $H$ for a representative of each connected component. The value of $H$ for the other vertices can then be obtained by determining to which connected component they belong. Both their and our energy functions satisfy $H((v_0\ot\vv_0)\ot(v_0\ot\vv_0))=0$, so they must be the same, even though their formulations differ. In this sense, Theorem \ref{theorem:perfcrys} gives a simpler, more explicit and unified formula for the $A_{n-1}^{(1)}$ energy function in \cite{Benkart}.

\medskip
Our proof of Theorem \ref{theorem:perfcrys} in Sections \ref{sec:toolperfcrys} and \ref{sec:proofcrys} relies on explicitly building paths in the crystal graph. We only treat the case $n \geq 3$, as $n=1$ and $n=2$ give crystals with a slightly different shape, and we already know that the theorem is true in these cases.

\medskip
Theorem \ref{theorem:perfcrys} gives a simple explicit expression for the energy function. Using the (KMN)$^2$ crystal base character formula \cite{KMN2a}, it allows us to relate the generating function $  G^P_{n}(q;b_0,\cdots,b_{n-1})$ of generalised Primc partitions and the generating function $G^C_{n}(\delta,\gamma;q;b_0,\cdots,b_{n-1})$ of generalised Capparelli partitions with the character of the irreducible highest weight module $L(\Lambda_0)$.

Unlike previous connections between character formulas and partition generating functions, where a specific specialisation (often the principal specialisation) was needed, here we give a \textbf{non-dilated character formula}.

\begin{theorem}\label{theorem:finalcharac}
Let $n$ be a positive integer, and let $\Lambda_0, \dots, \Lambda_{n-1}$ be the fundamental weights of $A_{n-1}^{(1)}.$
 By setting $e^{\wt v_i} = b_i$ and $e^{-\delta} = q$, we have the following identities:
 \begin{align*}
  G^P_{n}(q;b_0,\cdots,b_{n-1}) &=  \frac{e^{-\Lambda_0}\mathrm{ch}(L(\Lambda_0))}{(q;q)_{\infty}}, \\
  G^C_{n}(\delta,\gamma;q;b_0,\cdots,b_{n-1})&= e^{-\Lambda_0}\mathrm{ch}(L(\Lambda_0)).
 \end{align*}
\end{theorem}

This result gives an evaluation of the character of the irreducible highest weight module for the particular weight $\Lambda_0$, but we can extend our techniques to retrieve the characters for the other level $1$ weights 
of $\bar P_1^+$. 
\begin{theorem}\label{theorem:finalcharacter}
Let $n$ be a positive integer, and let $\Lambda_0, \dots, \Lambda_{n-1}$ be the fundamental weights of $A_{n-1}^{(1)}.$
 By setting $e^{\wt v_i} = b_i$ and $e^{-\delta} = q$, we have the following identities for any $\ell\in\{0,\ldots,n-1\}$:
 \begin{align*}
 G^P_{n}(q;b_0q,\cdots,b_{\ell-1}q,b_{\ell},\dots, b_{n-1})  &=  \frac{e^{-\Lambda_{\ell}}\mathrm{ch}(L(\Lambda_{\ell}))}{(q;q)_{\infty}},
 \\G^C_{n}(\delta,\gamma;q;b_0q,\cdots,b_{\ell-1}q,b_{\ell},\dots, b_{n-1})  &= e^{-\Lambda_{\ell}}\mathrm{ch}(L(\Lambda_{\ell})).
 \end{align*}
\end{theorem}
The case $\ell=0$ of Theorem \ref{theorem:finalcharacter} gives Theorem \ref{theorem:finalcharac}.

\medskip

As mentioned earlier, finding an expression of the character as a series with positive coefficients is an important problem. In \cite{KacPeterson}, using modular forms and string functions, Kac and Peterson gave a formula for $e^{-\Lambda}\mathrm{ch}(L(\Lambda))$ for all the irreducible highest weight level $1$ modules $L(\Lambda)$ of most classical types as a series in $\Z[[e^{-\alpha_0},e^{-\alpha_1},\cdots,e^{-\alpha_{n-1}}]]$ with obviously positive coefficients. This built on earlier work of Kac \cite{KacCharacter}, in which he proved the particular case where $M=L(\Lambda_0)$ in $A_n^{(1)},$ $D_n^{(1)},$ and $E_n^{(1)}.$

 In \cite{BW15}, Bartlett and Warnaar  used Hall-Littlewood polynomials to give explicitly positive formulas for the characters of certain highest weight modules of the affine Lie algebras $C_n^{(1)}$, $A_{2n}^{(2)}$, and $D_{n+1}^{(2)}$, which also led to generalisations for the Macdonald identities in types $B_n^{(1)},$ $C_n^{(1)},$ $A_{2n-1}^{(2)},$ $A_{2n}^{(2)}$, and $D_{n+1}^{(2)}.$ However their approach failed to give a formula for the case $A_{n}^{(1)}.$
 Using Macdonald-Koornwinder theory, Rains and Warnaar \cite{WarnaarRains} later found additional character formulas for these types, together with new Rogers-Ramanujan type identities.
 
 In \cite{GOW16}, Griffin, Ono, and Warnaar obtained a limiting Rogers-Ramanujan type identity for the principal specialisation of the character of some particular weights $(m-k)\Lambda_0+k\Lambda_1$ in $A_{n}^{(1)}.$ 
On the other hand, Meurman and Primc \cite{Meurman2} treated the case of all levels of $A_1^{(1)}$ via vertex operator algebras.

Here, using our non-dilated character formula from Theorem \ref{theorem:finalcharacter}, we recover, for all $\ell\in\{0,\dots,n-1\}$, the character formula of Kac-Peterson. Moreover, we give a new explicit expression for $e^{-\Lambda_{\ell}}\mathrm{ch}(L(\Lambda_{\ell}))$ as a sum of $(n-1)!$ series with positive coefficients  $\Z[[e^{-\alpha_0},e^{-\alpha_1},\cdots,e^{-\alpha_{n-1}}]]$, each of which are infinite product generating functions for partitions into distinct parts with congruence conditions.

\begin{theorem}\label{theorem:formulaexp}
Let $n$ be a positive integer, and let $\Lambda_0, \dots, \Lambda_{n-1}$ be the fundamental weights of $A_{n-1}^{(1)}.$
For all $\ell\in\{0,\ldots,n-1\}$, we have %in $\Z[[e^{-\alpha_0},e^{-\alpha_1},\cdots,e^{-\alpha_{n-1}}]]$ that
\begin{align}
&e^{-\Lambda_{\ell}}\mathrm{ch}(L(\Lambda_{\ell})) \nonumber
\\&=\frac{1}{(e^{-\delta};e^{-\delta})_{\infty}^{n-1}}\sum_{\substack{s_1, \dots, s_{n-1}\in \Z\\s_0=s_n=0}} e^{-s_{\ell}\delta}\prod_{i=1}^{n-1} e^{s_i\alpha_i}e^{s_i(s_{i+1}-s_{i})\delta} \label{eq:KacPet}
\\&=\left(\prod_{i=1}^{n-1}\frac{\left(e^{-i(i+1)\delta};e^{-i(i+1)\delta}\right)_{\infty}}{(e^{-\delta};e^{-\delta})_{\infty}}\right)\sum_{\substack{r_1, \dots, r_{n-1}:\\0 \leq r_j \leq j-1\\r_n=0}} e^{-r_l\delta}\prod_{i=1}^{n-1} e^{r_i\alpha_i}e^{r_i(r_{i+1}-r_i)\delta} \nonumber
\\& \qquad \qquad \qquad \times \left(-e^{(ir_{i+1}-(i+1)r_i-\frac{i(i+1)}{2}-\ell\chi(i\geq l>0))\delta+\sum_{j=1}^i j\alpha_j};e^{-i(i+1)\delta}\right)_{\infty} \label{eq:newcharacter}
\\& \qquad \qquad \qquad \times \left(-e^{((i+1)r_i-ir_{i+1}-\frac{i(i+1)}{2}+\ell\chi(i\geq l>0))\delta-\sum_{j=1}^i j\alpha_j};e^{-i(i+1)\delta}\right)_{\infty}, \nonumber
\end{align}
where $\delta= \alpha_0 + \cdots + \alpha_{n-1}$ is the null root.
\end{theorem}
The character formula \eqref{eq:KacPet} is, up to a change of variables, a reformulation of the Kac-Peterson character formula for the type $A^{(1)}_{n-1}$ given in \cite[p.217]{KacPeterson}.
%For example, when $\ell=0$, the index set in their formula is exactly the $\Z$-lattice with generators $\alpha_i (i\in \{1,\cdots,s\})$, and the scalar product  $(.|.)$ on this set satisfies $\frac{1}{2}(\alpha_i|\alpha_j)= a_{i,j}$, where the $a_{i,j}$ are the coefficients of the generalized Cartan matrix $A$.
Thus, our partition identity Theorem \ref{theorem:main2}, combined with Theorem \ref{theorem:finalcharacter}, makes the connection between the KMN$^2$ crystal base character formula and the Kac-Peterson character formula.

The \textit{principal specialisation} \cite[Chapter 10]{Kac} for the affine type $A_{n-1}^{(1)}$ consists in transforming the generators with 
\begin{equation*}
  e^{-\alpha_i}\mapsto  q \quad \text{for all }i\in\{0,\ldots,n-1\}.
\end{equation*}
In that case, we have a natural transformation $b_i := q^ib_0$ and 
a dilated version of the character formula can be deduced from Theorems \ref{theorem:main2} and \ref{theorem:finalcharacter}.
\begin{cor}
Let $n$ be a positive integer, and let $\Lambda_0, \dots, \Lambda_{n-1}$ be the fundamental weights of $A_{n-1}^{(1)}.$
For all $\ell\in \{0,\cdots,n-1\}$, the principal specialisation of $e^{-\Lambda_{\ell}}\mathrm{ch}(L(\Lambda_{\ell}))$, denoted by $\mathbb{F}_{\mathds{1}}(e^{-\Lambda_{\ell}}\mathrm{ch}(L(\Lambda_{\ell})))$, is the generating function of the classical integer partitions with no parts divisible by $n$ : 
 \begin{align*}
  \mathbb{F}_{\mathds{1}}(e^{-\Lambda_{\ell}}\mathrm{ch}(L(\Lambda_{\ell})))&= (q^n;q^n)\times G_n^P(q^n;q^nb_0,\cdots,q^{n+\ell-1}b_0,q^{\ell},\cdots,q^{n-1}b_0)\\
  &= (q^n;q^n)\times [x^0]\left(\prod_{i=0}^{\ell-1}(-q^{-i}b_0^{-1}x;q^n)_{\infty}(-q^{n+i}b_0x^{-1};q^n)_{\infty} \right.\\
  &\qquad\qquad\qquad\qquad\times \left.\prod_{i=\ell}^{n-1}(-q^{n-i}b_0^{-1}x;q^n)_{\infty}(-q^{i}b_0x^{-1};q^n)_{\infty}\right)\\
  &= (q^n;q^n)\times [x^0](-q^{1-\ell}b_0^{-1}x;q)_{\infty}(q^{\ell}b_0x^{-1};q)_{\infty}\\
  &= \frac{(q^n;q^n)}{(q;q)_\infty}.
 \end{align*}
\end{cor}
In this particular case, we recover the principal specialisation of the Weyl-Kac character formula \cite{Kac}.

\medskip
The remainder of this paper is organised as follows. In Section \ref{sec:reptheory}, we recall the necessary definitions and theorems about representation theory and crystal bases. In Section \ref{sec:groundedpartitions}, we define \textit{grounded partitions}, which will play a key role in obtaining a non-specialised character formula. In Section \ref{sec:An1}, we define the $A_{n-1}^{(1)}$ crystals related to our difference condition/energy function $\Delta.$ In Section \ref{sec:proofchar}, we prove our character formulas (Theorems \ref{theorem:finalcharac}, \ref{theorem:finalcharacter}, and \ref{theorem:formulaexp}) assuming that $\Delta$ is an energy function for our crystal. Finally, in Sections \ref{sec:toolperfcrys} and \ref{sec:proofcrys}, we prove that this is indeed the case, by constructing some paths on the crystal graph.

\section{Basics on Crystals}
\label{sec:reptheory}
In this section, we recall the definitions and basic theorems from crystal base theory which are necessary for our purpose.
We refer to the book \cite{HK}, which we consider to be a good summary of the basic theory of Kac-Moody algebras \cite[Chapter 2]{HK}, quantum groups \cite[Chapter 3]{HK} and crystal bases \cite[Chapters 4, 10]{HK}. For a more combinatorial approach and more emphasis on the finite dimensional case, we refer the reader to \cite{SchillingBump}.

Throughout this section, $n$ is a fixed positive integer.

\subsection{Cartan datum and quantum affine algebras}
A square matrix $A = \bigl( a_{i,j}\bigr)_{i,j\in \I}$ is said to be a generalised Cartan matrix  if $A$ has the
following properties:   
\begin{itemize}
\item for all $i \in \I$, $a_{i,i} = 2,$ 
\item for all $i \neq j$ in $\I$, $a_{i,j} \in \Z_{\leq 0},$
\item $a_{i,j} = 0$ if and only if $a_{j,i} = 0$,
\end{itemize}
Moreover, if there exists a diagonal matrix $D$ with positive integer coefficients such that $DA$ is symmetric, then $A$ is said to be symmetrisable.
In addition, if the rank of the matrix $A$ is $n-1$, then $A$ is said to be of affine type.
In this paper, we always assume that this is the case.

Let us consider such a matrix $A$.  Let $P^\vee$ be a free abelian group of rank $n+1$ with $\Z$-basis $\{h_0, \dots, h_{n-1},d\}:$
$$P^\vee = \Z h_0 \oplus \Z h_1 \oplus \cdots \oplus \Z h_{n-1}
\oplus \Z d.$$
We call $P^\vee$  the {\it dual weight lattice}. The complexification $\h = \mathbb C \ot_{\Z} P^\vee$ is called the \textit{Cartan subalgebra}. The linear functionals $\alpha_i$ and $\Lambda_i$ ($i \in \I$)
on $\h$ given by

\begin{equation}
\label{eq:rootcoroot}
\begin{array}{cccc} \langle h_j,\alpha_i \rangle: =  \alpha_i(h_j) = a_{j,i}& \qquad &\langle d, \alpha_i \rangle: = \alpha_i(d) = \delta_{i,0}& \\
\langle h_j, \Lambda_i \rangle: = \Lambda_i(h_j) = \delta_{i,j}& \qquad &\langle d, \Lambda_i \rangle: = \Lambda_i(d) = 0 &\quad (i,j\in \I) \end{array}
\end{equation}
are respectively the {\it simple roots} and {\it fundamental weights}. Let $\h^*$ be the dual space of $\h.$ We denote by  $\Pi = \{\alpha_i \mid i \in \I\} \subset \h^*$ the set of simple roots,
and define $\Pi^\vee = \{h_i \mid i \in \I\}\subset \h$ to be the set of \textit{simple coroots}. We also set
$$P = \{\lambda \in \mathfrak h^* \mid \lambda(P^\vee) \subset \Z\}$$
 to be the \textit{weight lattice}. It contains the set of {\it dominant integral weights}
$$
 P^+ = \{\lambda \in P \mid \lambda(h_i) \in \Z_{ \geq 0} \text{ for all i } \in \I\}.
$$
\m
The quintuple $(A,\Pi,\Pi^\vee,P,P^\vee)$ is said to be a \textit{Cartan datum} for the Cartan matrix $A$.
The \textit{Kac-Moody affine Lie algebra} $\widehat \g$ attached to this datum is the Lie algebra with generators $e_i,f_i \ (i \in \I)$ and $h \in P^\vee$, with the following defining relations
(\cite[Definition~2.1.3]{HK}):
\begin{enumerate}
 \item $[h,h']=0$ for all $h,h'\in P^\vee$,
 \item $[e_i,f_j]= \delta_{ij}h_j$,
 \item $[h,e_i]=\alpha_i(h)e_i$ for all $h\in P^\vee$,
 \item $[h,f_i]=-\alpha_i(h)f_i$ for all $h\in P^\vee$,
 \item $(\textrm{ad} e_i)^{1-a_{i,j}}e_j= (\textrm{ad} f_i)^{1-a_{i,j}}f_j =0$ for $i \neq j$,
\end{enumerate}
where $\textrm{ad} x : y\mapsto [x,y]$.

We also define the \textit{coroot lattice}
$$\bar P^\vee = \Z h_0 \oplus \Z h_1 \oplus \cdots \oplus \Z h_{n-1},$$
 and its complexification $\bar \h = \mathbb C \ot_{\Z}\bar P^\vee$.  The
 $\Z$-submodule
$$\Z \Lambda_0 \oplus \Z \Lambda_1 \oplus \cdots \oplus \Z \Lambda_{n-1}$$
of  $P$ is called the lattice of  {\it classical weights} and is denoted by $\bar P.$

\begin{remark}
By \eqref{eq:rootcoroot}, for all $j \neq 0,$ we have
$$\alpha_j = \sum_{i=0}^{n-1}a_{i,j} \Lambda_i \ \in \bar P.$$
We will denote by $\ov \alpha_0$ the restriction of $\alpha_0$ to $\bar P.$
\end{remark}

 Let
  \,$\bar P^+
: = \sum_{i= 0}^n  \Z_{ \geq 0} \Lambda_i$ \, denote the corresponding set of  dominant weights.  \m

The center 
$$\Z c = \{h\in P^\vee: \langle h,\alpha_i \rangle = 0 \text{ for all i } \in \I\}$$
of the affine Lie algebra $\widehat \g$
is one-dimensional and generated by the {\it canonical central element} $c$, where
$$c=c_0 h_0 + \cdots + c_{n-1} h_{n-1}.$$

The space of imaginary roots 
$$ \Z \delta = \{\lambda\in P: \langle h_i,\lambda \rangle = 0 \text{ for all i } \in \I\}$$
of $\widehat \g$ is also one-dimensional, generated by the {\it null root} $\delta$, where
$$\delta = d_0 \alpha_0 + d_1 \alpha_1 + \cdots + d_{n-1} \alpha_{n-1}.$$
The vector $^{t}(d_0,d_1,\dots, d_{n-1})
\in \mathbb C^{n}$ spans the kernel of the Cartan matrix $A$.
The {\it level \  $\ell$} of a dominant weight $\lambda \in P^+$ is given by the expression
$\langle c, \lambda \rangle : = \lambda(c) = \ell$. 

For  any $k \in \Z$ and an indeterminate $q$,  let us set
$$[k]_q  = \frac{q^k- q^{-k}}{q-q^{-1}}.$$
We also set $[0]_q\,! = 1$ and for $k \geq 1$, $[k]_q\,! = [k]_q[k-1]_q \cdots [1]_q$. For $m \geq k \geq 0$, define
$$\left\langle\begin{array}{c} m \\ k \end{array}\right\rangle_{q} = \frac{[m]_q\,!}{[k]_q\,!\,[m-k]_q\,!}.$$

\medskip
 We now have all the definitions necessary to introduce quantum affine Lie algebras.
\begin{defn}\label{defn:qalg}\cite[Definition 3.1.1]{HK}
The {\it quantum affine algebra} $U_q(\widehat \g)$ associated with the Cartan datum $(A,\Pi,\Pi^\vee, P, P^\vee)$
is the associative algebra with unit element over $\mathbb C(q)$ (where $q$ is an indeterminate)
with generators $e_i,\,f_i \ (i \in \I)$, and $q^h \ (h \in P^\vee)$,  satisfying the defining relations:
\begin{itemize}
\item[\rm{(1)}] $ q^0 = 1, \ q^hq^{h'} = q^{h+h'}$  \ \ \  for $h,h' \in P^\vee$,
\item[\rm{(2)}] $q^h e_i q^{-h} = q^{\alpha_i(h)}e_i$ \ \ \ for $h \in P^\vee$, $i \in \I$,
\item[\rm{(3)}] $q^h f_i q^{-h} = q^{-\alpha_i(h)}f_i$ \ \ \  for $h \in P^\vee$, $i \in \I$,
\item[\rm{(4)}] $e_if_j - f_je_i = \displaystyle{\delta_{i,j}\frac{K_i-K_i^{-1}}{q_i-q_i^{-1}}}$ \ \ \  for $i,j \in \I$,
\item[\rm{(5)}] $\displaystyle{\sum_{k=0}^{1-a_{i,j}} \left\langle\begin{array}{c} 1-a_{i,j} \\ k \end{array}\right\rangle_{q_i}
e_i^{1-a_{i,j}-k} e_j e_i^k = 0}$  \ \ \  for  $i \neq j$,
\item[\rm{(6)}] $\displaystyle{\sum_{k=0}^{1-a_{i,j}} \left\langle\begin{array}{c} 1-a_{i,j} \\ k \end{array}\right\rangle_{q_i}
f_i^{1-a_{i,j}-k} f_j f_i^k = 0}$  \ \ \ for  $i \neq j$.
 \end{itemize}
Here $q_i = q^{s_i}$ and $K_i = q^{s_i h_i}$, where $D=diag(s_i : i \in \{0, \dots, n-1\})$ is a symmetrising matrix of $A$.
\end{defn}
%For any $\lambda \in P$, the \textit{Verma module} is defined as the quotient 
%\begin{equation}
% M(\lambda) = U_q(\widehat g)/J(\lambda)\,,
%\end{equation}
%where $J(\lambda)$ is the ideal of $U_q(\widehat g)$ generated by $e_i$ $(i\in \I)$ and $q^h-q^{\lambda(h)}1$ $(h\in P^\vee)$. Let us set $u_\lambda = 1+J(\lambda)$. The \textit{Verma module} is a \textit{highest weight module} with highest weight $\lambda$
%and \textit{highest weight vector} $u_\lambda$, as it satisfies the following properties:
%\begin{enumerate}
% \item $U_q(\widehat g)u_\lambda = M(\lambda)$,
% \item $q^h u_\lambda  = q^{\lambda(h)}u_\lambda$ for all $h\in P^\vee$,
% \item $e_iu_\lambda =0$ for all $i\in \I$.
%\end{enumerate}
%This module has a unique maximal submodule $N(\lambda)$, and the quotient $L(\lambda)=M(\lambda)/N(\lambda)$ is called the \textit{irreducible highest weight module} with highest weight $\lambda$. 
%\m
\begin{defn}\label{defn:subalgebra}
The \textit{quantum affine algebra} $U_q'(\widehat \g)$ is the subalgebra of $U_q(\widehat \g)$ generated by $e_i, f_i, K_i^{\pm 1}\ (i \in \I)$.
\end{defn}
Contrarily to $U_q(\widehat \g)$, the quantum affine algebra $U_q'(\widehat \g)$ admits some non-trivial finite-dimensional irreducible modules.

\subsection{Integrable modules, highest weight modules and character formula}
We are now ready to define irreducible highest weight modules and characters.
\begin{defn}
Let $\g$ be a Lie algebra with bracket $[\cdot,\cdot]$, and let $V$ be a vector space. Then $V$ is a \textit{$\g$-module} if there is a bilinear map  $\g \times V \rightarrow V$, denoted by $(x,v) \mapsto x \cdot v$, satisfying for all $x,y \in \g$ and all $v \in V$:
$$[x,y] \cdot v = x \cdot (y \cdot v) - y \cdot (x \cdot v). $$
\end{defn}

A subspace $W$ of a $\g$-module $V$ is called a \emph{submodule} of $V$ if for all $x \in \g$, $x \cdot W \subseteq W.$

A $\g$-module $V$ is said to be \emph{irreducible} if its only submodules are $V$ and $0$.

The notion of modules extends naturally from an affine Lie algebra $\widehat\g$ to its quantum affine algebra $U_q(\widehat \g)$.

\m
\begin{defn}
A $U_q(\widehat \g)$-module $M$ is said to be \emph{integrable} if it satisfies the following properties:
\begin{itemize}
\item[{(a)}] $M$ has a weight space decomposition: \
$M = \bigoplus_{\lambda \in P} M_\lambda$,  where
$M_\lambda = \{v \in M \mid q^h\cdot v = q^{\lambda(h)}v$ for all $h \in P^\vee \}$;
\item[{(b)}] there are finitely many $\lambda_1,\dots, \lambda_k \in P$ such that
$\mathrm{wt}(M) \subseteq \Omega(\lambda_1) \cup \dots \cup \Omega(\lambda_k)$,\
where $\mathrm{wt}(M) = \{\lambda \in P \mid M_\lambda \neq 0\}$ and
$\Omega(\lambda_j) = \{\mu \in P \mid \mu \in \lambda_j+ \sum_{i\in \I} \Z_{\leq 0} \alpha_i\}$;
\item[{(c)}]  the elements $e_i$ and $f_i$ act locally nilpotently on $M$ for all $i \in \I$.
\end{itemize} 
We denote by $\mathcal O^q_{\hbox {\rm \small  int}}$ the category of integrable $U_q(\widehat \g)$-modules.
\end{defn}

\m
For all $\lambda\in P$, a module of \textit{highest weight} $\lambda$ is an integrable module such that:
\begin{itemize}
\item[{(a)}]$\mathrm{wt}(M) \subseteq \Omega(\lambda)$; 
\item[{(b)}]$\dim M_{\lambda}=1$;
\item[{(b)}]$M = U_q(\widehat \g) M_{\lambda}$.
\end{itemize}
For all $\lambda\in P$, up to isomorphism, there exists a unique highest weight module which is \textit{irreducible}. We denote by $L(\lambda)$ the \textit{irreducible highest weight} $U_q(\widehat \g)$-module of highest weight $\lambda$.
\begin{defn}
Let $M$ be an integrable module such that $\dim M_{\lambda}<\infty$ for all $\lambda\in \mathrm{wt}(M).$ The \textit{character} of $M$ is defined by 
\begin{equation}\label{eq:character}
 \mathrm{ch} (M) = \sum_{\lambda\in \mathrm{wt}(M)} \dim M_{\lambda} \cdot e^{\lambda},
\end{equation}
where the $e^{\lambda}$'s  are formal basis elements of the group algebra $\mathbb C[\mathfrak{h}^*]$, with the multiplication defined by $e^{\lambda}e^{\mu}= e^{\lambda+\mu}$.
\end{defn}  
When $M$ is a highest weight module of highest weight $\lambda$, its character satisfies
$$
e^{-\lambda}\mathrm{ch} (M) = \sum_{\mu\in \mathrm{wt}(M)} \dim M_{\lambda} \cdot e^{\mu-\lambda} \quad \in \quad \Z_{\geq 0}[[e^{-\alpha_i}, i\in \I]].
$$
All these definitions on modules also hold in the case of the $\widehat \g$-modules $M'$, where the weight spaces are given by
$M'_\lambda = \{v \in M' \mid h\cdot v = \lambda(h)v$ for all $h \in P^\vee \}$. Thus, looking at the generators of the weight spaces, for a fixed weight $\lambda\in P$, the \textit{irreducible highest weight} $\widehat \g$-module can be identified with the \textit{irreducible highest weight} $U_q(\widehat \g)$-module, and we have equality of characters.
\subsection{Crystal bases}
The crystal base theory was developed independently by Kashiwara \cite{Kas90} and Lusztig \cite{Lusz90} to study the category $\mathcal O^q_{\hbox {\rm \small  int}}$ of integrable $U_q(\widehat g)$-modules.
If $M$ is a module in the category $\mathcal O^q_{\hbox {\rm \small  int}}$, then for each $i \in \I$,  a weight vector $u \in M_\lambda$
can be written uniquely in the form $u = \sum_{k=0}^N  f_i^{(k)} u_k$,  for some $N\geq 0$ and 
$u_k \in M_{\lambda+k\alpha_i} \cap \ker e_i$  for  all $k=0,1,\dots,N$, with  
$f_i^{(k)} = f_i^k/ ([k]_{q_i}!)$.    
The {\it Kashiwara operators} $\tilde e_i$ and $\tilde f_i$, for $i \in\I$, are then defined as follows:
\begin{equation}
\label{eq:kashiwara}
\tilde e_i u = \sum_{k=1}^N f_i^{(k-1)}u_k,   \qquad \ \   \tilde f_i u = \sum_{k=0}^N
f_i^{(k+1)}u_k.
\end{equation}

Crystal bases will be seen as bases at $q=0.$ To do so, let us define the \textit{localisation} of $\mathbb C[q]$ at $q=0$ by $\mathbb A_0 = \{ f = g/h \mid g,h  \in \mathbb C[q], \  h(0) \neq 0 \}.
$ 
\begin{defn}\cite[Definition 4.2.2]{HK}  Assume that $M$ is a $U_q(\widehat \g)$-module in the category $\mathcal O^q_{\hbox {\rm \small  int}}$.  A free
$\mathbb A_0$-submodule $\mathcal L$ of $M$ is a {\it crystal lattice} if
\begin{itemize}
\item[{\rm (i)}] $\mathcal L$ generates $M$ as a vector space over $\mathbb C(q)$;
\item[{\rm (ii)}]  $\mathcal L = \bigoplus_{\lambda \in P} \mathcal L_\lambda$ \  where \
$\mathcal L_\lambda = M_\lambda \cap \mathcal L$;
\item[{\rm (iii)}]  $\tilde e_i \mathcal L \subset \mathcal L$ \  and \
$\tilde f_i \mathcal L \subset \mathcal L,$\ for all $i \in \I$.
\end{itemize}  \end{defn}
Since the operators $\tilde e_i$ and $\tilde f_i$ preserve the lattice $\mathcal L$,
they also define operators on the quotient $\mathcal L/q \mathcal L$.

\begin{defn}\cite[Definition 4.2.3]{HK}  A {\it crystal base} for a $U_q(\widehat \g)$-module $M \in \mathcal O^q_{\hbox {\rm \small  int}}$ is a pair
$(\mathcal L, \B)$   such that
\begin{itemize}
\item[{\rm (1)}] $\mathcal L$ is a crystal lattice of $M$;
\item[{\rm (2)}]  $\B$ is a $\mathbb C$-basis of $\mathcal L/q \mathcal L \cong \mathbb C \ot_{\mathbb A_0} \mathcal L$;
\item[{\rm(3)}]  $\B = \sqcup_{\lambda \in P} \B_\lambda$, \ where \
$\B_\lambda = \B \cap (\mathcal L_\lambda/ q \mathcal L_\lambda)$;
\item[{\rm(4)}]  $\tilde e_i \B_{\lambda} \subset \B_{\lambda+\alpha_i} \cup \{0\}$ \  and \
$\tilde f_i \B_{\lambda} \subset \B_{\lambda-\alpha_i} \cup \{0\}$ for all $i \in \I$;
\item[{\rm(5)}]  $\tilde f_i b = b'$ \  if and only if \  $b = \tilde e_i b',$ \  for \  $b,b' \in \B$ \ and \ $i\in \I$.  \end{itemize}
\end{defn}

To each module $M \in \mathcal O^q_{\hbox {\rm \small  int}}$, one can associate a corresponding crystal base $(\mathcal L, \B)$, which is unique up to isomorphism \cite[Chapter 5]{HK}. Therefore, from now on, we will refer to ``the'' crystal base of $M.$ 

Furthermore, the 
\textit{crystal graph} associated to $(\mathcal L, \B)$ can be defined as follows. The set of vertices is $\B$, and the oriented edges are built as follows:
$$ b \xrightarrow[]{\,\,\, i \,\,\,} b' \quad \text{if and only if}\quad \tilde f_i b = b' \text{ (or equivalently } \tilde e_i b' = b) .$$
\begin{rem}
When $\tilde f_i b =0$ (resp. $\tilde e_i b =0$), then there is no edge labelled $i$ coming out of $b$ (resp. arriving in $b$).
\end{rem}
The crystal graph can be viewed as a combinatorial data of the module $M$.
\m
For $i \in \I$,  let us define functions $\varepsilon_i, \varphi_i: \B \rightarrow \Z$ as follows:
$$
\begin{array}{cc}
&\varepsilon_i(b) = \max\{ k  \geq 0 \mid \tilde e_i^k b \in \B\}, \\
&\varphi_i(b) =  \max\{ k  \geq 0 \mid \tilde f_i^k b \in \B\}. \end{array}
$$
Thus  $\varepsilon_i(b)$ is the length of the longest chain of $i$-arrows ending at $b$ in the crystal graph, and $\varphi_i(b)$ is the length of the longest chain of
$i$-arrows starting from $b$.  Furthermore, we have $\varphi_i(b) -\varepsilon_i(b) = \lambda(h_i)$ for
all $b \in \B_\lambda$.  Thus, by setting $\mathrm{wt} b = \lambda$,
\begin{equation}
\label{eq:epsilonphi}
\varepsilon(b) = \sum_{i=0}^{n-1} \varepsilon_i(b) \Lambda_i,
\qquad \text{and} \qquad \varphi(b) = \sum_{i=0}^{n-1} \varphi_i(b) \Lambda_i,
\end{equation}
we then have $ \wt b = \varphi(b) -\varepsilon(b)$ for all $b \in \B_\lambda$, where $\wt b$ is the projection of $\mathrm{wt}b$ on $\ov{P}$. Also, by the definition of the weight vectors $u_k$ in the Kashiwara operators \eqref{eq:kashiwara}, we have for all $b\in \B$ such that
$\tilde e_i b \neq 0$,
\begin{equation}\label{eq:weightalpha}
 \mathrm{wt} \tilde e_i b - \mathrm{wt} b = \alpha_i.
\end{equation}

Let us now introduce the notion of crystal.
\begin{defn}\cite[Definition 4.5.1]{HK} 
Let $A=(a_{i,j})_{0 \leq i,j \leq n-1}$ be a Cartan matrix with associated \textit{Cartan datum} $(A,\Pi,\Pi^\vee,P,P^\vee).$ A \textit{crystal} associated with $(A,\Pi,\Pi^\vee,P,P^\vee)$ is a set $\B$ together with maps
\begin{align*}
\mathrm{wt} : \B &\longrightarrow P,\\
\tilde e_i, \tilde f_i : \B &\longrightarrow \B \cup \{0\} \qquad \quad\ (i \in \I),\\
\varepsilon_i,\varphi_i : \B &\longrightarrow \Z \cup \{- \infty\} \qquad (i \in \I),
\end{align*}
satisfying the following properties for all $i \in \I$:
\begin{enumerate}
\item $\varphi_i(b)=\varepsilon_i(b)+ \langle h_i, \mathrm{wt} (b) \rangle,$
\item $\mathrm{wt}(\tilde e_ib)= \mathrm{wt}b+\alpha_i$ if $\tilde{e_i}b \in \B,$
\item $\mathrm{wt}(\tilde f_ib)= \mathrm{wt}b-\alpha_i$ if $\tilde{f_i}b \in \B,$
\item $\varepsilon_i(\tilde e_ib)= \varepsilon_i(b)-1$ if $\tilde{e_i}b \in \B,$
\item  $\varphi_i(\tilde e_ib)= \varphi_i(b)+1$ if $\tilde{e_i}b \in \B,$
\item $\varepsilon_i(\tilde f_ib)= \varepsilon_i(b)+1$ if $\tilde{f_i}b \in \B,$
\item  $\varphi_i(\tilde f_ib)= \varphi_i(b)-1$ if $\tilde{f_i}b \in \B,$
\item $\tilde{f_i}b=b'$ if and only if $b= \tilde{e_i}b'$ for $b,b' \in B,$
\item if $\varphi_i(b)= - \infty$ for $b \in \B,$ then $\tilde{e_i}b= \tilde{f_i}b=0.$
\end{enumerate}
\end{defn}
In particular, if $(\mathcal L, \B)$  is a crystal base, then $\B$ is a crystal.

\medskip
Let $\B_1$  and 
$\B_2$  be two crystals. % of  $U_q(\widehat \g)$.
A \textit{crystal morphism} between $\B_1$ and $\B_2$ is a map 
$\Psi: \B_1 \cup \{0\} \rightarrow \B_2 \cup
\{0\}$ 
such that 
\begin{itemize}
\item $\Psi(0) = 0$;
\item $\Psi$ commutes with $\mathrm{wt},\varepsilon_i,\varphi_i$ for all $i \in \I$;
\item for $b, b' \in \B_1$ such that
$\tilde f_i b = b'$ and $\Psi(b), \Psi(b') \in \B_2$, we have $\tilde
f_i \Psi(b)= \Psi(b')$, $\tilde e_i \Psi(b')=\Psi(b)$.
\end{itemize}
 A morphism $\Psi$ is said to be {\it strict} if it commutes with $\tilde e_i$, $\tilde f_i$ for all $i\in \I$.

\medskip

The theory of crystal bases behaves very nicely with respect to the tensor product of $\mathcal O^q_{\hbox {\rm \small  int}}$-modules, as can be seen in the next theorem.
\begin{theorem}\label{theorem:tensorrules}\cite[Theorem 4.4.1]{HK}
Let $M_1, M_2 \in \mathcal O_{\hbox {\rm \small  int}}$, and let $(\mathcal L_1,\B_1),(\mathcal L_2,\B_2)$ be the corresponding
crystal bases. We set $\mathcal L = \mathcal L_1 \ot_{{\mathbb A}_0} \mathcal L_2$ and $\B
= \B_1 \ot \B_2 \equiv \B_1 \times \B_2$. 
Then $(\mathcal L, \B)$ is a crystal base
of $M_1 \ot_{\mathbb C(q)} M_2$, with 
\begin{equation}\label{eq:tensrul}
\begin{array}{cc}
&\tilde e_i (b_1 \ot b_2) = \begin{cases} \tilde e_i b_1 \ot b_2 \quad & \hbox{\rm if} \ \ \varphi_i(b_1)  \geq \varepsilon_i(b_2), \\
b_1 \ot \tilde e_i b_2 \quad & \hbox{\rm if} \ \ \varphi_i(b_1) < \varepsilon_i(b_2), \end{cases}\\
&\tilde f_i (b_1 \ot b_2) = \begin{cases} \tilde f_i b_1 \ot b_2 \quad & \hbox{\rm if} \ \ \varphi_i(b_1) > \varepsilon_i(b_2), \\
b_1 \ot \tilde f_i b_2 \quad & \hbox{\rm if} \ \ \varphi_i(b_1) \leq \varepsilon_i(b_2), \end{cases}
\end{array} \end{equation}
where $b_1 \ot 0 = 0\ot b_2 = 0$ for all $b_1\in\B_1$ and $b_2\in \B_2$. Furthermore, we have
\begin{align*}
 \mathrm{wt}(b_1\ot b_2)&= \mathrm{wt} b_1+ \mathrm{wt} b_2,\\
 \varepsilon_i (b_1\ot b_2) &= \max\{\varepsilon_i(b_1),\varepsilon_i(b_1)+\varepsilon_i(b_2)-\varphi_i(b_1)\},\\
 \varphi_i (b_1\ot b_2) &= \max\{\varphi_i(b_2),\varphi_i(b_1)+\varphi_i(b_2)-\varepsilon_i(b_2)\}.
\end{align*}
\end{theorem}

\medskip
The last tool we need in this paper is the notion of energy function, defined as follows. 
\begin{defn} \cite[Definition 10.2.1]{HK} Let $M \in \mathcal O^q_{\hbox {\rm \small  int}}$ be a module, and $(\mathcal L,\B)$ be the corresponding crystal base.
An {\it energy function} on $\B \ot \B$ is a map $H: \B \ot \B \rightarrow
\Z$ satisfying
\begin{equation}\label{eq:ef} H\left(\tilde e_i (b_1 \ot b_2)\right)
= \begin{cases} H(b_1 \ot b_2) & \qquad \hbox{\rm if} \ \ i \neq 0, \\
H(b_1 \ot b_2) + 1 & \qquad \hbox{\rm if} \ \ i = 0 \ \hbox{\rm and} \ \varphi_0(b_1)  \geq \varepsilon_0(b_2)  \\
H(b_1 \ot b_2) - 1 & \qquad \hbox{\rm if} \ \ i = 0 \ \hbox{\rm and} \ \varphi_0(b_1)< \varepsilon_0(b_2),
\end{cases} \end{equation}
for all $i \in \I$ and $b_1,b_2$ with $\tilde e(b_1 \ot b_2) \neq 0$.
\end{defn}
By definition, in the crystal graph of $\B \ot \B$, the value of $H(b_1\ot b_2)$, when it exists, determines all the values $H(b'_1\ot b'_2)$ for vertices $b'_1\ot b'_2$ in the same connected component as $b_1\ot b_2$.
Note that the conditions \eqref{eq:ef} are equivalent to the following:
\begin{equation}\label{eq:enfct}
\begin{array}{c}
H\left(\tilde e_i (b_1 \ot b_2)\right)
= \begin{cases}
H(b_1 \ot b_2) + \chi(i=0) & \qquad \hbox{\rm if} \ \ \varphi_i(b_1)  \geq \varepsilon_i(b_2)  \\
H(b_1 \ot b_2) - \chi(i=0) & \qquad \hbox{\rm if} \ \ \varphi_i(b_1)< \varepsilon_i(b_2),
\end{cases} \\
H(\tilde f_i (b_1 \ot b_2))
= \begin{cases}
H(b_1 \ot b_2) - \chi(i=0) & \qquad \hbox{\rm if} \ \ \varphi_i(b_1)  >\varepsilon_i(b_2)  \\
H(b_1 \ot b_2) + \chi(i=0) & \qquad \hbox{\rm if} \ \ \varphi_i(b_1)\leq  \varepsilon_i(b_2).
\end{cases} 
\end{array}
\end{equation}

Figure  \ref{fig:exA1} gives the crystal graph $\B$ of the vector representation of $A_{1}^{(1)}$ \cite[10.5.2]{HK}, the tensor product $\B\ot\B$, and an energy function $H$ on $\B\ot\B$.
\begin{fig}
\label{fig:exA1}
\[\]
\begin{center}
\begin{tikzpicture}[thick,scale=0.8, every node/.style={scale=0.8}]
%%crystal graph B
\draw (-2,0) node {$\B$ :};
\draw (0,0) node {$0$};
\draw (1.5,0) node {$1$};
\draw (-2.5,-1.5) node {$\B\ot\B$ :};
\draw (0,-1.5) node {$0$};
\draw (1.5,-1.5) node {$1$};
\draw (0,-3) node {$1$};
\draw (1.5,-3) node {$1$};
\draw (-0.8,-3) node {$0$};
\draw (2.3,-3) node {$1$};
\draw (-0.8,-1.5) node {$0$};
\draw (2.3,-1.5) node {$0$};
\foreach \x in {0,1}
\foreach \y in {0,1.5,3}
\draw (-0.2+1.5*\x,0.2-\y)--(0.2+1.5*\x,0.2-\y)--(0.2+1.5*\x,-0.2-\y)--(-0.2+1.5*\x,-0.2-\y)--(-0.2+1.5*\x,0.2-\y);
\foreach \x in {0,1}
\foreach \y in {1.5,3}
\draw (-1+3.1*\x,0.2-\y)--(-0.6+3.1*\x,0.2-\y)--(-0.6+3.1*\x,-0.2-\y)--(-1+3.1*\x,-0.2-\y)--(-1+3.1*\x,0.2-\y);
\foreach \x in {0,1.5}
\draw [thick,->] (0.3,-\x)--(1.2,-\x);
\draw [thick,<-] (0.3,-3)--(1.2,-3);
\draw [thick,<-] (-0.4,-1.7)--(-0.4,-2.8);
\draw [thick,->] (1.9,-1.7)--(1.9,-2.8);

\foreach \x in {0,1}
\foreach \y in {1.5,3}
\draw (-0.4+2.3*\x,-\y) node {$\ot$};

\draw [->] (1.4,-0.3) arc (-60:-120:1.3);

\draw (0.75,0.15) node {\footnotesize{$1$}};
\draw (0.75,-1.35) node {\footnotesize{$1$}};
\draw (0.75,-3.2) node {\footnotesize{$0$}};
\draw (2,-2.25) node {\footnotesize{$1$}};
\draw (-0.5,-2.25) node {\footnotesize{$0$}};
\draw (0.75,-0.6) node {\footnotesize{$0$}};

\draw (5,-1.5) node {$H$:};

\foreach \x in {0,1}
\draw (5.8+0.8*\x,-1.3 )--(6.2+0.8*\x,-1.3)--(6.2+0.8*\x,-1.7)--(5.8+0.8*\x,-1.7)--(5.8+0.8*\x,-1.3);
\draw (6.4,-1.5) node {$\ot$};
\draw (6,-1.5) node {$i$};
\draw (6.8,-1.5) node {$j$};
\draw ( 7.5 ,-1.55) node {$\mapsto$};
\draw ( 7.8 ,-1.5) node[right] {$\chi(i\geq j)$};
\end{tikzpicture}
\end{center}
\end{fig}
\subsection{Perfect crystals} The theory of perfect crystals was developed by Kang, Kashiwara, Misra, Miwa, Nakashima, and Nakayashiki \cite{KMN2a,KMN2b} to study the irreducible highest weight modules over quantum affine algebras. Indeed, perfect crystals provide a construction of the crystal base $\B(\lambda)$ of any irreducible $U_q(\widehat \g)$-module
$L(\lambda)$ corresponding to a classical weight  $\lambda \in \bar P^+$. 
We call \textit{affine crystal} an abstract crystal associated with an affine Cartan datum $(A,\Pi,\Pi^\vee,P,P^\vee)$ (quantum algebra $U_q(\widehat \g)$), while the term \textit{classical crystal} is used for an abstract crystal
associated to the classical Cartan datum $(A,\Pi,\Pi^\vee, \bar P,\bar P^\vee)$ (quantum algebra $U'_q(\widehat \g)$ defined in Definition \ref{defn:subalgebra}). 

All the theorems in this section are due to Kang, Kashiwara, Misra, Miwa, Nakashima, and Nakayashiki, but we give references to the book \cite{HK} for reader's convenience.

Let us start by defining perfect crystals.
\begin{defn}\label{defn:pc}\cite[Definition 10.5.1]{HK}   For a positive integer $\ell$, a finite
classical crystal $\B$ is said to be a {\it perfect crystal of
level $\ell$}  for the quantum affine algebra $U_q(\widehat \g)$ if
\begin{itemize}
\item[{\rm (1)}]  there is a finite-dimensional $U_q'(\widehat \g)$-module
with a crystal base whose crystal graph is isomorphic to
\, $\B$;
\item[{\rm (2)}]  $\B \otimes \B$ \,  is connected;
\item[{\rm (3)}]  there exists a classical weight \,$\lambda_0$\,
such that

$$\mathrm{wt}(\B) \subset \lambda_0 + \frac{1}{d_0} \sum_{i \neq 0} \Z_{\leq 0} \alpha_i
\quad \hbox{\rm and} \quad  |\B_{\lambda_0}| = 1;$$

\item[{\rm (4)}]  for any $b \in \B$, we have
$$\langle c,\varepsilon(b)\rangle
= \sum_{i=0}^{n-1} \varepsilon_i(b) \Lambda_i(c)  \geq \ell;$$

\item[{\rm (5)}]   for each $\lambda \in \bar P_\ell^+ := \{ \mu \in \bar P^+  \mid \langle c, \mu \rangle = \ell\}$, there exist unique
vectors $b^\lambda$ and $b_\lambda$  in $\B$ such that
$\varepsilon(b^\lambda) = \lambda$ and $\varphi(b_\lambda) = \lambda$.
\end{itemize}
\end{defn}
In the remainder of this section, we fix a perfect crystal $\B.$

The maps $\lambda \mapsto \varepsilon(b_\lambda)$ and $\lambda\mapsto \varphi(b^\lambda)$ then define two bijections on $\bar P_\ell^+$. 

As a consequence of the last condition, for any $\lambda \in \bar P^+_\ell$, the vertex operator theory \cite[(10.4.4)]{HK} leads to a natural crystal isomorphism
\begin{gather}\label{defn:vertex} \mathcal B(\lambda) \ \iso \  \mathcal B(\varepsilon(b_{\lambda})) \ot \mathcal B \\
  \quad u_{\lambda} \, \mapsto \ \ u_{\varepsilon(b_{\lambda})}  \ot b_{\lambda}  .\nonumber 
\end{gather}

\begin{defn}  For $\lambda \in \bar P^+_\ell$,   the {\it ground state path of weight} $\lambda$ is the tensor product
 $${\p}_\lambda = \,\bigl(g_k)_{k=0}^\infty \,= \  \  \cdots \ot g_{k+1} \ot g_k \ot \cdots \ot g_1 \ot g_0,$$
 where the elements $g_k \in \B$ are such that 
 \begin{equation}\label{eq:lamb}\begin{array} {ccc}
\lambda_0 = \lambda &\qquad g_0 = b_\lambda &  \\
\lambda_{k+1} = \varepsilon(b_{\lambda_k})
 &\qquad \, g_{k+1} = b_{\lambda_{k+1}}  & \qquad  \hbox{\rm for all}\ \
k  \geq 0\,\cdot\end{array} \end{equation}
A tensor product $\p = (p_k)_{k=0}^\infty =  \cdots \ot p_{k+1} \ot p_k \ot \cdots \ot p_1 \ot p_0$ of elements $p_k \in \B$ is said to be a $\lambda$-{\it path} if
$p_k = g_k$ for $k$ large enough.  
\end{defn}

Iterating the isomorphism \eqref{defn:vertex}, we obtain
$$\begin{array}{ccccccc}
  \mathcal B(\lambda) &\iso& \mathcal B(\lambda_1) \ot \mathcal B &\iso& \mathcal B(\lambda_2) \ot \mathcal B \ot \mathcal B &\iso &\cdots    \\  
  u_\lambda &\mapsto& u_{\lambda_1} \ot g_0 & \mapsto & u_{\lambda_2} \ot g_1 \ot g_0 &\mapsto &\ \ \cdots  ,
\end{array}$$
which gives a natural bijection stated in the next theorem. 
  
 \begin{theorem}\label{theorem:cryiso}\cite[Theorem 10.6.4]{HK} 
 Let $\lambda \in  \bar P^+_\ell$.
 Then there is a crystal isomorphism
 \begin{align*}
 \B(\lambda)  &\iso \mathcal P(\lambda)\\
 u_\lambda &\mapsto \p_\lambda
 \end{align*}
 between the crystal base $\B(\lambda)$ of $L(\lambda)$ and the set $\mathcal P(\lambda)$ of $\lambda$-paths.
 \end{theorem}  

We describe the crystal structure of $\mathcal P(\lambda)$ as follows \cite[(10.48)]{HK}.   
For any $\p = (p_k)_{k=0}^\infty \in \mathcal P(\lambda)$, let $N \geq 0$ be the smallest integer such that  $p_k = g_k$ for all $k  \geq N$. We then set
\begin{align*}
\wt \p &= \lambda_N + \sum_{k=0}^{N-1} \wt p_k,\\
\tilde e_i \p\  &= \ \ \cdots \ot g_{N+1} \ot \tilde e_i\left( g_N \ot \cdots \ot p_0 \right), \\
\tilde f_i \p \ &= \ \ \cdots \ot g_{N+1} \ot \tilde f_i\left( g_N \ot \cdots \ot p_0 \right), \\
\varepsilon_i(\p) &= \max\left( \varepsilon_i(\p')-\varphi_i(g_N), 0\right),  \\
\varphi_i(\p) &= \varphi_i(\p') +  \max\left(\varphi_i(g_N)-\varepsilon_i(\p'), 0\right),
\end{align*} 
where $\p' := p_{N-1} \ot \cdots \ot p_1 \ot p_0$, and
$\wt$ is viewed as the classical weight of an element
of $\B$ or $\mathcal P(\lambda)$. \m
 
The explicit expression for the affine weight  $\mathrm{wt} \p$
in $P$ is given in the following theorem, which is known as the (KMN)$^2$ crystal base character formula, and plays a key role in connecting characters with partition generating functions.

\begin{theorem}\label{theorem:wtchar}\cite[Theorem 10.6.7]{HK}  
Let $\lambda \in \bar P^+_{\ell}$, let $H$ be an energy function on $\B \ot \B$, and let  $\p = (p_k)_{k=0}^\infty \in \mathcal P(\lambda)$.
Then the weight of $\p$ and the character of the irreducible highest weight $U_q(\widehat \g)$-module
$L(\lambda)$ are given by the following expressions:
\begin{align}\label{eq:firsteq}
\mathrm{wt} \p &= \lambda + \sum_{k=0}^\infty \left(\ov{\hbox{\rm wt}p}_k -\ov{\hbox{\rm wt}g}_k\right)  -\left( \sum_{k=0}^\infty (k+1)\Big(H(p_{k+1} \ot p_k) - H(g_{k+1}\ot g_k)\Big)\right) \frac{\delta}{d_0},  \nonumber \\
         &= \lambda + \sum_{k=0}^\infty \left(\ov{\hbox{\rm wt}p}_k -\ov{\hbox{\rm wt}g}_k\right)   -\left( \sum_{l=k}^\infty (H(p_{l+1} \ot p_l) - H(g_{l+1}\ot g_l))\right) \frac{\delta}{d_0},\\
\rm ch(L(\lambda)) &= \sum_{\p \in \mathcal P(\lambda)}  e^{\ov{\hbox{\rm wt}\p}}.  \label{eq:secondeq}
\end{align}
\end{theorem}

 A specialisation of Theorem \ref{theorem:wtchar} gives the following corollary.
\begin{cor}\label{cor:difcond}  Suppose that $\Ll$ is such that $b_{\Ll}=b^{\Ll}= g$, and set $ H(g \ot g)=0$.
Then  $\ov{\rm{wt}}g = 0$, $g_k=g$ for all $k\in \mathbb{Z}_{\geq 0}$, and we have  
\begin{equation*}%\label{eq:secondeq}
\mathrm{wt}\p = \lambda + \sum_{k=0}^\infty \ov{\hbox{\rm wt}\p}_k-\left( \sum_{l=k}^\infty H(p_{l+1} \ot p_l)\right) \frac{\delta}{d_0}.
\end{equation*}
\end{cor}
This is the main result which we will use in the next section to connect crystal base theory to integer partitions. 

Note that the $d_0$ in Theorem \ref{theorem:wtchar} and Corollary \ref{cor:difcond} did not appear in \cite{HK} and the original work \cite{KMN2a}, but this was a typo which was fixed later in \cite{KKM} for example. However, it did not affect their results, as for most classical types (including $A_{n-1}^{(1)}$), we have $d_0=1$.

\section{Perfect crystals and grounded partitions}
\label{sec:groundedpartitions}
To make the connection between our combinatorial partition identities and character formulas, we introduce in this section a new type of coloured partitions: grounded partitions.

Let $\C$ be a set of colours, and let $\Z_{\C} = \{k_c : k \in \Z, c \in \C\}$ be the set of coloured integers. First, we relax the condition that parts of (coloured) partitions have to be in non-increasing order.
\begin{defn}
Let $\gg$ be a binary relation defined on $\Z_{\C}$.
A \textit{generalised coloured partition} with relation $\gg$ is a finite sequence $(\pi_0,\ldots,\pi_s)$ of coloured integers, where for all $i \in \{0, \dots, s-1\},$ $\pi_i\gg\pi_{i+1}.$
\end{defn}
In the following, $c(\pi_i) \in \C$ denotes the colour of the part $\pi_i$. The quantity $|\pi|=\pi_0+\cdots+\pi_s$ is the weight of $\pi$, and $C(\pi) = c(\pi_0)\cdots c(\pi_s)$ is its colour sequence. 

\begin{remark}
The binary relation is not necessarily an order. When $\gg$ is a strict total order, we can easily check 
that every finite set of coloured parts defines a classical coloured partition, by ordering the parts. In the same way, for a large total order, the generalised coloured partitions are finite multi-sets of coloured integers.
\end{remark}

Let us choose a particular colour $\co$. We now define grounded partitions, which are directly related to ground state paths.
\begin{defn}
A \textit{grounded partition} with ground $\co$ and relation $\gg$ is a non-empty generalised coloured partition $\pi=(\pi_0,\ldots,\pi_s)$ with relation $\gg$, such that $\pi_s = 0_{\co}$, and when $s>0$, $\pi_{s-1}\neq 0_{\co}.$

\noindent Let $\Ppp$ denote the set of such partitions.
\end{defn}
In the following, we explicitly write $\pi=(\pi_0,\ldots,\pi_{s-1},0_{\co})$. The trivial partition in $\Ppp$ is then $(0_{\co})$.  

\begin{ex}For the set of classical integer partitions $\pi=(\pi_1,\ldots,\pi_s)$, where parts satisfy $\pi_1\geq \cdots \geq \pi_s>0$, the empty partition is such that $s=0$.
This set is in bijection 
with the set $\Pp_{c}$ of grounded coloured partitions with only one colour $c$, defined by the relation 
$$k_c\gg l_c \text{ if and only if } k-l\geq 0.$$
The bijection is given by 
$$(\pi_1,\ldots,\pi_s) \mapsto ((\pi_1)_c,\ldots,(\pi_s)_c,0_c),$$
where the empty partition $\emptyset$ corresponds to the grounded partition $(0_c)$.
\end{ex}

In the remainder of this section, we make the connection between grounded partitions and crystal base theory. Let us fix a weight $\Ll \in \bar P^+_{\ell}$ such that  $b_{\Ll}=b^{\Ll}=g$, and assume that $ H(g \ot g)=0$. Let $\C_{\B}=\{c_b: \,b\in \B\}$ be the set of colours indexed by $\B$. 
We define the binary relation $\gtrdot$ on $\Z_{\C_{\B}}$ by 
\begin{equation}\label{eq:restr}
k_{c_b}\gtrdot k'_{c_{b'}} \text{ if and only if } k-k'= H(b'\ot b).
\end{equation}
This relation leads to the following.

\begin{prop} \label{prop:flat1}
Let $\phi$ be the map between $\lambda$-paths and grounded partitions defined as follows:
\begin{equation*}
\phi : \quad \p\mapsto(\pi_0,\ldots,\pi_{s-1},0_{\co}),
\end{equation*}
where $\p=(p_k)_{k\geq 0}$ is a $\Ll$-path in $\Pp(\Ll)$, $s\geq 0$ is the unique non-negative integer such that $p_{s-1}\neq g$ and $p_k=g$ for all $k\geq s$, 
and for all $k\in \ssss$, the part $\pi_k$ has colour $c_{p_k}$ and size
\begin{equation*}
\sum_{l=k}^{s-1} H(p_{k+1}\ot p_k).
\end{equation*}
Then $\phi$ is a bijection between $\Pp(\Ll)$ and $\Pppp$. Furthermore, by taking $c_{b}=e^{\wt b}$, we have for all $\pi\in \Pppp$,  
\begin{equation}\label{eq:paca}
e^{-\Ll+\mathrm{wt}\phi^{-1}(\pi)}  = C(\pi)e^{-\frac{\delta}{d_0} |\pi|}.
\end{equation} 
\end{prop}
\m
\begin{proof}
 It is easy to see that $\phi(\p)$ belongs to $\Pppp$, since by \eqref{eq:restr} we have $\pi_k\gtrdot \pi_{k+1} $ for all $k\in\ssss$, and 
$p_{s-1}\neq g$ implies that $\pi_{s-1}\neq 0_{\co}$. Note that the ground state path $\cdots \ot g \ot \ot g \ot g$ is associated to $(0_{\co})$.

Let us now give the inverse bijection. Start with $\pi\in (\pi_0,\ldots,\pi_{s-1},0_{c_{g}})\in \Pppp$, different from $(0_{\co})$, 
with colour sequence $c_{p'_0}\cdots c_{p'_{s-1}} \co$. Recall that $\pi_s=0_{\co}$. We set $\phi^{-1}(\pi) = (p_k)_{k\geq 0}$, where $p_k= g$ for all $k\geq s$ and $p_k= p'_k$ for all $k\in \ssss$.
\begin{itemize}
 \item We first show that $p_{s-1}\neq g$. Assume for the purpose of contradiction that $p_{s-1}=g$. By \eqref{eq:restr}, we know that $\pi_{s-1}\gtrdot 0_{\co}$ if and only if $$\pi_{s-1}-0_{\co}= H(p_s\ot  p_{s-1})=H(g\ot g) = 0,$$
i.e. if and only if $\pi_{s-1}= 0_{\co}$. This contradicts the fact that $\pi_{s-1}\neq 0_{\co}$. 
 \item By \eqref{eq:restr}, we also have, for all $k\in \ssss$, $\pi_k-\pi_{k+1}=H(p_{k+1}\ot p_k).$ Therefore
$$\pi_k = \pi_k-0_{\co} = \sum_{l=k}^{s-1}\pi_l-\pi_{l+1}= \sum_{l=k}^{s-1} H(p_{l+1}\ot p_l).$$
\end{itemize}
With what precedes, we have $\phi(\phi^{-1}(\pi))=\pi$ and $\phi^{-1}(\phi(\p))=\p$. We obtain \eqref{eq:paca} by Corollary \ref{cor:difcond} and by observing that
\[\pi_k =\sum_{l=k}^{s-1} H(p_{l+1}\ot p_l)=\sum_{l=k}^\infty H(p_{l+1}\ot p_l),\]
since $H(p_{l+1}\ot p_l)=H(g \ot g) =0$ for all $l\geq s$.
\end{proof}
\begin{ex}
Let us consider the energy matrix \eqref{Primcmatrix4} given by Primc for the case $A_1^{(1)}$, with the correspondence $c_{v_j\ot \vv_i} = a_ib_j$ for all $i,j\in \{0,1\}$. Let us set the ground $g=v_0\ot \vv_0$ corresponding to the classical weight $\Lambda_0$, so that $\co = a_0b_0$. 
\begin{itemize}
\item The ground state path $\p_{\Lambda_0} = \cdots \ot (v_0\ot \vv_0) \ot (v_0\ot \vv_0)$ corresponds to the partition $\phi(\p_{\Lambda_0})=(0_{a_0b_0})$.
\item For $\p = \cdots \ot (v_0\ot \vv_0) \ot (v_0\ot \vv_0) \ot ( v_1\ot \vv_0) \ot (v_0\ot \vv_1) \ot (v_1\ot \vv_1),$ we have 
$$\phi(\p) = (3_{a_1b_1},3_{a_1b_0},1_{a_0b_1},0_{a_0b_0}).$$
\end{itemize}
\end{ex}
The next proposition allows us to describe the set $\Ppp$ of grounded partitions for the relation $\gg$ defined by
\begin{equation}\label{eq:relation}
 k_{c_b}\gg k'_{c_{b'}} \text{ if and only if } k-k'\geq H(b'\ot b).
\end{equation}
We refer to this relation as \textit{minimal difference conditions}.
One can view the partitions of $\Pppp$ as the partitions of $\Ppp$ such that the differences between consecutive parts are minimal.
%We then refer to  the partitions of $\Pppp$ as some minimal partitions of $\Ppp$.
Note that contrarily to $\Pppp$, the set $\Ppp$ has some partitions $\pi = (\pi_0,\ldots,\pi_{s-1},0_{\co})$ such that $c(\pi_{s-1})= \co$. 
For this reason, the set $\Pppp$ is not exactly the set of all minimal partitions of $\Ppp$, but is still related to it.

\begin{prop}\label{prop:diff}
Recall that $\Pp_{\co}$ is the set of grounded partitions where all parts have colour $\co$. There is a bijection $\Phi$ between $\Ppp$ and $\Pppp\times \Pp_{\co}$, such that if $\Phi(\pi) =(\mu,\nu)$, then $|\pi|=|\mu|+|\nu|$, and by setting $\co=1$, we have $C(\pi)=C(\mu)$.
\end{prop}
\begin{proof}
We set $\Phi(0_{\co})=((0_{\co}),(0_{\co}))$. Let us now consider any $\pi=(\pi_0,\ldots,\pi_{s-1},0_{c_{g}})\in \Ppp$, 
different from $(0_{\co})$, with colour sequence $c_{p'_0}\cdots c_{p'_{s-1}} \co$, and build $\Phi(\pi) =(\mu,\nu).$ Recall that $\pi_{s-1}\neq \pi_s= 0_{\co}$. Let us set  $ \p= (p_k)_{k\geq 0}$, with $p_k= g$ 
for all $k\geq s$ and $p_k= p'_k$ for all $k\in \ssss$, and set
$$r=\max\{k\in \{0,\ldots,s\}: p_{k-1}\neq g\}.$$
Since $p_k=g$ for all $k\geq r$, with the convention $\co=1$, we obtain that 
$C(\pi)=c_{p_0}\cdots c_{p_{s-1}} = c_{p_0}\cdots c_{p_{r-1}}$.
Note that $r=0$ if and only if all the parts of $\pi$ have colour $\co$. We set $\mu=(\mu_0,\ldots,\mu_{r-1},0_{\co})= \phi(\p)$. By Proposition \ref{prop:flat1}, for all $k\in \{0,\ldots,r-1\}$, the part $\mu_k$ is coloured by $c_{p_k}$ and has size
\[\sum_{l=k}^{r-1} H(p_{l+1}\ot p_l).\] 
Let us now build $\nu = (\nu_0,\ldots,\nu_{t-1},0_{\co})\in \Pp_{\co}$, where  $c(\nu_k)=\co$ and $\nu_k>0$ for all $k\in \{0,\cdots,t-1\}$. We distinguish two different cases.
\begin{itemize}
\item If $r<s$, then we set $t=s$ and $\nu = (\nu_0,\ldots,\nu_{s-1},0_{\co})$, where 
$$\begin{cases}
\nu_k&=\ \pi_k-\mu_k\quad \text{for } k \in \{0,\ldots,r-1\},\\
\nu_k&=\ \pi_k \qquad \quad \  \text{for } k\in \{r,\ldots,s-1\}.
\end{cases}$$
By \eqref{eq:relation}, the sequence
$(\nu_0,\ldots,\nu_{r-1})$ is non-increasing. Moreover the fact that $H(g\ot g)=0$ and $\pi_{s-1}\neq 0_{\co}$ implies that $\nu_{s-1}>0$, and $(\nu_0,\ldots,\nu_{s-1})$ is a non-increasing sequence of positive integers. Finally, let us check that $\nu_{r-1} \geq \nu_r.$ We have
\begin{align*}
\nu_{r-1}-\nu_r &= \pi_{r-1}-\pi_r -\mu_{r-1}&\\
&\geq H(p_{r}\ot p_{r-1})- H(p_{r}\ot p_{r-1}) &\text{by \eqref{eq:relation}}\\
&\geq 0.&
\end{align*}
Thus $(\nu_0,\ldots,\nu_{s-1})$ is indeed a non-increasing sequence of positive integers. 
\item By definition we have $r \leq s$, so the only other possible case is $r=s$. As before, $(\pi_0-\mu_0, \dots , \pi_s-\mu_s)$ is a non-increasing sequence of non-negative integers, now with $\pi_s-\mu_s = 0-0=0$. We then set 
$$t=\min\{k\in \{0,\ldots,s\}: \pi_k=\mu_k\},$$
and $\nu_k = \pi_k-\mu_k$ for all $k\in \{0,\ldots,t-1\}$.
\end{itemize}
\bi
Observe that for $\Phi(\pi) = (\mu,\nu)$, with $\pi = (\pi_0,\ldots,\pi_{s-1},0_{\co})$, $\mu = (\mu_0,\ldots,\mu_{r-1},0_{\co})$ and $\nu = (\nu_0,\ldots,\nu_{t-1},0_{\co})$, we always have $s= \max\{r,t\}$, and by adding $s- \min\{r,t\}$ parts $0_{\co}$ at the end of the shorter partition, we have $\pi_k = \mu_k+\nu_k$ and $c(\pi_k)=c(\mu_k)$ for all $k\in \{0,\ldots,s-1\}$. 
\bi
The map $\Phi^{-1}$ from $\Pppp\times \Pp_{\co}$ to $\Ppp$ simply consists in
adding the parts of $\mu = (\mu_0,\ldots,\mu_{r-1},0_{\co}) \in \Pppp$ to those of $\nu = (\nu_0,\cdots,\nu_{t-1},0_{\co}) \in \Pp_{\co}$ to obtain a grounded partition $\pi \in \Ppp$ in the following way: 
\begin{itemize}
 \item if $t\leq r$, then $\pi_k$ has size $\mu_k + \nu_k$ and colour $c(\mu_k)$, where we set $\nu_k=0$ for all $k\in\{t,\cdots,r-1\}$, and we obtain the partition
 \[\pi = (\pi_0,\cdots,\pi_{r-1},0_{\co}),\]
 \item if $t>r$, the first $r$ parts are defined as in the case $t\leq r$, and the remaining parts are $\pi_k = \nu_k$ for all $k\in \{r,\ldots,t-1\}$ with colour $\co$, and we obtain the partition
 \[\pi = (\pi_0,\cdots,\pi_{t-1},0_{\co}).\]
\end{itemize}
\end{proof}
\begin{exs}
Let us consider the energy matrix \eqref{Primcmatrix4} given Primc for the case $A_1^{(1)}$, and set
$\co=a_0b_0$. We give three examples of the previous bijection, for the different cases $t<r$, $t=r$, and $t>r$.
\begin{itemize}
\item Case $t<r$: Let $\pi = (10_{a_0b_0},7_{a_1b_0},5_{a_0b_1},3_{a_1b_1},2_{a_0b_0},1_{a_1b_0},0_{a_0b_0})$. Given the minimal difference conditions in \eqref{Primcmatrix4}, our bijection gives
$$\mu = (6_{a_0b_0},5_{a_1b_0},3_{a_0b_1},3_{a_1b_1},2_{a_0b_0},1_{a_1b_0},0_{a_0b_0})\quad\text{and}\quad \nu=(4_{a_0b_0},2_{a_0b_0},2_{a_0b_0},0_{a_0b_0}).$$
\item Case $t=r$: Let $\pi = (10_{a_0b_0},7_{a_1b_0},5_{a_0b_1},4_{a_1b_1},3_{a_0b_0},2_{a_1b_0},0_{a_0b_0})$. We have 
$$\mu = (6_{a_0b_0},5_{a_1b_0},3_{a_0b_1},3_{a_1b_1},2_{a_0b_0},1_{a_1b_0},0_{a_0b_0})\quad\text{and}\quad \nu=(4_{a_0b_0},2_{a_0b_0},2_{a_0b_0},1_{a_0b_0},1_{a_0b_0},1_{a_0b_0},0_{a_0b_0}).$$
\item Case $t>r$: Let $\pi = (8_{a_0b_0},5_{a_1b_0},3_{a_0b_1},2_{a_1b_1},1_{a_0b_0},1_{a_0b_0},0_{a_0b_0})$. We obtain
$$\mu = (4_{a_0b_0},3_{a_1b_0},1_{a_0b_1},1_{a_1b_1},0_{a_0b_0})\quad\text{and}\quad \nu=(4_{a_0b_0},2_{a_0b_0},2_{a_0b_0},1_{a_0b_0},1_{a_0b_0},1_{a_0b_0},0_{a_0b_0}).$$
\end{itemize} 

\end{exs}
\medskip
We are now able to give a character formula in terms of generating functions for grounded partitions.

\begin{theorem} \label{theorem:formchar}
Setting $q=e^{-\frac{\delta}{d_0}}$ and $c_b=e^{\wt b}$ for all $b\in \B$, we have $\co=1$, and the character of the irreducible highest weight $U_q(\widehat \g)$-module
$L(\lambda)$ is given by the following expressions:
\begin{align*}
\sum_{\pi\in \Pppp} C(\pi)q^{|\pi|} &= e^{-\Ll}\rm{ch}(L(\Ll)),\\
 \sum_{\pi\in \Ppp} C(\pi)q^{|\pi|} &= \frac{e^{-\Ll}\rm{ch}(L(\Ll))}{(q;q)_{\infty}}.
\end{align*}
\end{theorem}
\begin{proof}
By Propositon \ref{prop:flat1} and \eqref{eq:secondeq},
$$\sum_{\pi\in \Pppp} C(\pi)q^{|\pi|} = \sum_{\p \in \mathcal P(\lambda)} e^{-\Ll}e^{\mathrm{wt}\p}= e^{-\Ll}\rm ch(L(\lambda)).$$
By Corollary \ref{cor:difcond}, $\wt g =0.$ Thus $c_{g} = e^0=1$, and  Proposition \ref{prop:diff} yields
$$\sum_{\pi\in \Ppp} C(\pi)q^{|\pi|} = \frac{1}{(q;q)_{\infty}} \sum_{\pi\in \Pppp} C(\pi)q^{|\pi|}= \frac{e^{-\Ll}\rm{ch}(L(\Ll))}{(q;q)_{\infty}}.$$
\end{proof}
By this theorem, the characters of irreducible highest weight modules of level $\ell$ can be computed as generating functions for some grounded partitions.
It is the key that connects generalised Primc partitions to the characters of irreducible highest weight modules of level $1$ for the affine Lie algebra $A_{n-1}^{(1)}$.

\section{Perfect crystal of type $A_{n-1}^{(1)}$: tensor product of the vector representation and its dual}
\label{sec:An1}
We now describe the perfect crystal $\Bb$ used in Theorem \ref{theorem:perfcrys}.
Throughout this section, we fix an integer $n \geq 3.$

Consider the Cartan datum for the matrice $A=(a_{ij})_{i,j\in \I}$ where for all $i,j \in \I,$
\begin{equation}\label{eq:datum}
a_{ij} = 2\delta_{i,j}-\chi(i-j\equiv \pm 1 \mod n).
\end{equation}
It corresponds to the affine type $A_{n-1}^{(1)}$ \cite[10.1.1]{HK}.
We then have the corresponding canonical central element $c$ and null root  $\delta$, which are expressed in the following way:
\begin{equation}\label{eq:centnull}\begin{array}{cc} &c = h_0 + h_1 + \cdots + h_{n-1}, \\
&\delta = \alpha_0 + \alpha_1 + \cdots + \alpha_{n-1}.
\end{array}\end{equation}
Any dominant integral weight $\lambda = k_0\Lambda_0+\cdots+k_{n-1}\Lambda_{n-1}\in \bar P^+$ has level 
\begin{equation*}
 \langle c,\lambda\rangle = k_0+\cdots+k_{n-1}.
\end{equation*}
Thus, the set of classical weights of level $1$ is exactly $\bar P^+_1 = \{\Lambda_i: i\in \I\}$, the set of  fundamental weights.

\medskip
A perfect crystal of level 1 is given by the crystal graph in Figure \ref{fig:B} \cite[11.1.1]{HK}.

\begin{fig}
%\[\]
\label{fig:B}
 \[\]
\begin{center}
\begin{tikzpicture}
%%crystal graph B
\draw (-2,0) node{$\B$ :};
\draw (0,0) node {$0$};
\draw (1.5,0) node {$1$};
\draw (5,0) node {$n-2$};
\draw (7,0) node {$n-1$};
\draw (3.15,0) node {$\cdots$};
\foreach \x in {0,1} 
\draw (-0.2+1.5*\x,0.2)--(0.2+1.5*\x,0.2)--(0.2+1.5*\x,-0.2)--(-0.2+1.5*\x,-0.2)--(-0.2+1.5*\x,0.2);
\foreach \x in {0,1} 
\draw (-0.45+5+2*\x,0.2)--(0.45+5+2*\x,0.2)--(0.45+5+2*\x,-0.2)--(-0.45+5+2*\x,-0.2)--(-0.45+5+2*\x,0.2);
\draw [thick, ->] (0.3,0)--(1.2,0);
\draw (0.75,0.15) node {\footnotesize{$1$}};
\draw [thick, ->] (1.8,0)--(2.7,0);
\draw (2.25,0.15) node {\footnotesize{$2$}};
\draw [thick, ->] (3.6,0)--(4.5,0);
\draw (4.05,0.15) node {\footnotesize{$n-2$}};
\draw [thick, ->] (5.6,0)--(6.4,0);
\draw (0.75,0.15) node {\footnotesize{$1$}};
\draw (6,0.15) node {\footnotesize{$n-1$}};
\draw [->,thick] (6.85,-0.3) arc(-75:-105:13);
\draw (3.5,-0.9) node {\footnotesize{$0$}};
\end{tikzpicture}
\end{center}
\end{fig}
The $U'_q(\widehat \g)$-module corresponding to this crystal is called the \textit{vector representation} of $A_{n-1}^{(1)}$. The most important property of this crystal is the order in which 
 the arrows appear. The only purpose of labelling the vertices is to ease the calculations in the remainder of this paper.  
Noting that this crystal graph is cyclic, we identify $\I$ with the group $(\Z/n\Z,+)$. In this way, the crystal graph of $\B$ can be defined locally around each arrow $i$ as shown in Figure \ref{fig:Blocal}.
\begin{fig}
\label{fig:Blocal}
 \[\]
\begin{center}
\begin{tikzpicture}
%%crystal graph B
\draw (-2,0) node{$\B(\xrightarrow[]{\,\,i\,\,})$ :};
\draw (0,0) node {$i-1$};
\draw (2,0) node {$i$};
\draw (-0.4,0.2)--(0.4,0.2)--(0.4,-0.2)--(-0.4,-0.2)--(-0.4,0.2);
\draw (-0.2+2,0.2)--(0.2+2,0.2)--(0.2+2,-0.2)--(-0.2+2,-0.2)--(-0.2+2,0.2);
\draw [thick, ->] (0.5,0)--(1.7,0);
\draw (1.1,0.15) node {\footnotesize{$i$}};
\end{tikzpicture}
\end{center}
\end{fig}
\begin{remark}
For the type $A_1^{(1)},$ the Cartan matrice $A$ is defined differently and is given by 
\begin{equation*}
  \begin{pmatrix}
   2&-2\\-2&2
  \end{pmatrix}.
\end{equation*}
Nonetheless, the crystal graph of the vector representation behaves in the same way as in the case $n\geq 3$.
\end{remark}
For all $i \in \I,$ let $v_i$ be the element of $\B$ corresponding to the vertex labelled $i$. The functions of this crystal are given by the following relations:
\begin{equation} \label{eq:crysrel}  
\wt v_i = \Lambda_{i+1}-\Lambda_i \quad \text{for all } i\in \I,
\end{equation}
\begin{equation}\label{eq:fii} 
\left\lbrace
\begin{array}{rcl}
 \tilde f_i v_{i-1} &=& v_i \\
 \varphi_i v_{i-1} &=& 1 \\
\tilde f_i v_{j}=\varphi_i v_{j}&=& 0 \qquad\text{if}\quad j\neq i-1,
\end{array}\right.
\end{equation}
\begin{equation}\label{eq:ei}
\left\lbrace
\begin{array}{rcl}
\tilde e_i v_{i} &=& v_{i-1} \\
 \varepsilon_i v_{i} &=& 1 \\
\tilde e_i v_{j}=\varepsilon_i v_{j}&=& 0 \qquad\text{if}\quad j\neq i.
\end{array}\right.
\end{equation}
\m We note that for this crystal, the unique \textit{maximal weight} $\lambda_0$, as defined in Condition $(3)$ of Definition \ref{defn:pc}, is attained in $v_0$ (i.e. $\lambda_0= \wt v_0$). For all $i\in\I$, we have 
\begin{align*}
 \wt v_0 - \wt v_i &= \sum_{j=1}^i \wt v_{j-1}-\wt v_j\\
 &= \sum_{j=1}^i \alpha_j \quad\text{by \eqref{eq:weightalpha}}.
\end{align*}
%Here the roots $\alpha_i (i\in \I)$ are seen as classical weights in $\bar P$, and 
The fact that the null root vanishes on $\bar \h$ implies that in $\bar P$,
$\ov \alpha_0 = -(\alpha_1+\cdots + \alpha_{n-1})$. We also remark that the crystal $\B$ has a unique \textit{minimal weight}, attained in $v_{n-1}$ : 
\begin{align*}
 \wt v_i - \wt v_{n-1} &= \sum_{j=i+1}^{n-1} \wt v_{j-1}-\wt v_j\\
 &= \sum_{j=i+1}^{n-1} \alpha_j \quad\text{by \eqref{eq:weightalpha}}.
\end{align*}

Let us consider the dual $\B^\vee$ of $\B,$ which is the crystal obtained from $\B$ by reversing the edges in its graph, as shown on Figure \ref{fig:4.3}.
\begin{fig}
\label{fig:4.3}
 \[\]
\begin{center}
\begin{tikzpicture}
%%crystal graph B
\draw (-2,0) node{$\B^\vee$ :};
\draw (0,0) node {$0$};
\draw (1.5,0) node {$1$};
\draw (5,0) node {$n-2$};
\draw (7,0) node {$n-1$};
\draw (3.15,0) node {$\cdots$};
\foreach \x in {0,1} 
\draw (-0.2+1.5*\x,0.2)--(0.2+1.5*\x,0.2)--(0.2+1.5*\x,-0.2)--(-0.2+1.5*\x,-0.2)--(-0.2+1.5*\x,0.2);
\foreach \x in {0,1} 
\draw (-0.45+5+2*\x,0.2)--(0.45+5+2*\x,0.2)--(0.45+5+2*\x,-0.2)--(-0.45+5+2*\x,-0.2)--(-0.45+5+2*\x,0.2);
\draw [thick, <-] (0.3,0)--(1.2,0);
\draw (0.75,0.15) node {\footnotesize{$1$}};
\draw [thick, <-] (1.8,0)--(2.7,0);
\draw (2.25,0.15) node {\footnotesize{$2$}};
\draw [thick, <-] (3.6,0)--(4.5,0);
\draw (4.05,0.15) node {\footnotesize{$n-2$}};
\draw [thick, <-] (5.6,0)--(6.4,0);
\draw (0.75,0.15) node {\footnotesize{$1$}};
\draw (6,0.15) node {\footnotesize{$n-1$}};
\draw [<-,thick] (6.85,-0.3) arc(-75:-105:13);
\draw (3.5,-0.9) node {\footnotesize{$0$}};
\end{tikzpicture}
\end{center}
\end{fig}
Let $\vv$ denote the element of $\B^\vee$ corresponding to $v$ in $\B$. We then have the relations
\begin{equation}\label{eq:duality}
\wt \vv= - \wt v\,,\quad \tilde f_i \vv = (\tilde e_i v)^\vee\,,\quad \varphi_i \vv = \varepsilon_i v\,,\quad\tilde e_i \vv = (\tilde f_i v)^\vee \,\,\text{ and}\quad \varepsilon_i \vv = \varphi_i v.
\end{equation}
Recall that the duality is an involution, since by the previous equalities, we have 
\begin{equation}\label{eq:dualdual}
(\tilde f_i [(\vv)^\vee],\tilde e_i [(\vv)^\vee],\varphi_i [(\vv)^\vee],\varepsilon_i [(\vv)^\vee ])=(\tilde f_i [(\vv)^\vee],\tilde e_i [(\vv)^\vee] ,\varphi_i v,\varepsilon_i v ),  
\end{equation}
and the map $v \mapsto (\vv)^\vee$ is an isomorphism between $\B$ and $(\B^\vee)^\vee$. Thus $(\B^\vee)^\vee$ can be identified with $\B$.

The dual $\B^\vee$ is also a perfect crystal of level 1, as his maximal weight is attained in the dual $\vv_{n-1}$ of the minimal vertex $v_{n-1}$ of $\B$. 

Moreover, for two crystals $\B_1$ and $\B_2$, we have
\begin{equation}
\label{eq:dual_tensor_product}
(\B_1 \ot \B_2)^\vee = \B_2^\vee \ot \B_1^\vee.
\end{equation}

By Theorem \ref{theorem:tensorrules}, $\B\ot \B^\vee$ is a crystal for the tensor product of the vector representation of $A_{n-1}^{(1)}$ and its dual, and the tensor rules \eqref{eq:tensrul} on $\B\ot \B^\vee$ become
\begin{align*}
\tilde e_i (v_k \ot \vv_l) &= \begin{cases} \tilde e_i v_k \ot \vv_l \quad & \hbox{\rm if} \ \ \varphi_i(v_k)  \geq \varphi_i(v_l) \\
v_k \ot \tilde e_i \vv_l\quad & \hbox{\rm if} \ \ \varphi_i(v_k)  < \varphi_i(v_l) \end{cases},\\
\tilde f_i (v_k \ot \vv_l) &= \begin{cases} \tilde f_i v_k \ot \vv_l \quad & \hbox{\rm if} \ \ \varphi_i(v_k)  > \varphi_i(v_l) \\
v_k \ot \tilde f_i \vv_l \quad & \hbox{\rm if} \ \ \varphi_i(v_k)  \leq \varphi_i(v_l) \end{cases}.
\end{align*}
Using \eqref{eq:fii} and \eqref{eq:ei}, we can draw the corresponding crystal graph, given in Figure \ref{fig:crystal}. 

\newpage
\begin{fig}
\[\]
\label{fig:crystal}
\begin{center}
\begin{tikzpicture}
%%crystal graph B
\draw (-7,0) node {$\B\ot\B^\vee$ :};

\draw (-6,0) node[right]{$v_0\ot \vv_{n-1}$};
\draw (-3.5,0) node[right]{$v_1\ot \vv_{n-1}$};
\draw (1.5,0) node[right] {{$v_{n-2}\ot \vv_{n-1}$}};
\draw (4.5,0) node[right] {$v_{n-1} \ot \vv_{n-1}$};

\draw (-6,-1) node[right]{$v_0\ot \vv_{n-2}$};
\draw (-3.5,-1) node[right]{$v_1\ot \vv_{n-2}$};
\draw (1.5,-1) node[right] {{$v_{n-2}\ot \vv_{n-2}$}};
\draw (4.5,-1) node[right] {$v_{n-1} \ot \vv_{n-2}$};

\draw (-6,-3) node[right]{$v_0\ot \vv_2$};
\draw (-3.5,-3) node[right]{$v_1\ot \vv_2$};
\draw (1.5,-3) node[right] {{$v_{n-2}\ot \vv_2$}};
\draw (4.5,-3) node[right] {$v_{n-1} \ot \vv_2$};

\draw (-6,-4) node[right]{$v_0\ot \vv_{1}$};
\draw (-3.5,-4) node[right]{$v_1\ot \vv_{1}$};
\draw (1.5,-4) node[right] {{$v_{n-2}\ot \vv_{1}$}};
\draw (4.5,-4) node[right] {$v_{n-1} \ot \vv_{1}$};

\draw (-6,-5) node[right]{$v_0\ot \vv_0$};
\draw (-3.5,-5) node[right]{$v_1\ot \vv_0$};
\draw (1.5,-5) node[right] {{$v_{n-2}\ot \vv_0$}};
\draw (4.5,-5) node[right] {$v_{n-1} \ot \vv_0$};

\draw (-1,0) node[right]{$v_2\ot \vv_{n-1}$};
\draw (-1,-1) node[right]{$v_2\ot \vv_{n-2}$};
\draw (-1,-3) node[right] {{$v_2\ot \vv_2$}};
\draw (-1,-4) node[right] {{$v_2\ot \vv_1$}};
\draw (-1,-5) node[right] {$v_2 \ot \vv_0$};

\foreach \x in {0,1}
\draw [thick, ->] (-5.34+2.5*\x,-0.15)--(-5.34+2.5*\x,-0.85);
\foreach \x in {0,1,3.16,4.36}
\draw [thick, dashed, ->] (-5.34+2.5*\x,-1.15)--(-5.34+2.5*\x,-2.85);

\foreach \x in {0,1,2,4.36}
\draw (-5.34+2.5*\x,-0.5) node[left] {\tiny{$n-1$}};
\foreach \y in {0,3,4,5}
\draw (4,0.15-\y) node {\tiny{$n-1$}};

\foreach \x in {0,2,3.16,4.36}
\draw (-5.34+2.5*\x,-3.5) node[left] {\tiny{$2$}};
\foreach \y in {0,1,3,5}
\draw (-1.5,0.15-\y) node {\tiny{$2$}};

\foreach \x in {1,2,3.16,4.36}
\draw (-5.34+2.5*\x,-4.5) node[left] {\tiny{$1$}};
\foreach \y in {0,1,3,4}
\draw (-4,0.15-\y) node {\tiny{$1$}};

\foreach \y in {0,3,4}
\draw [thick, ->] (-5.34+5,-0.15-\y)--(-5.34+5,-0.85-\y);
\draw [thick, ->] (-5.34,-3.15)--(-5.34,-3.85);
\draw [thick, ->] (-5.34+2.5,-4.15)--(-5.34+2.5,-4.85);

\foreach \x in {0,1}
\foreach \y in {3,4}
\draw [thick, ->] (2.56+3*\x,-0.15-\y)--(2.56+3*\x,-0.85-\y);;
\draw [thick, ->] (2.56+3,-0.15)--(2.56+3,-0.85);

\foreach \y in {0,1,3,4}
\draw [thick, ->] (-4.4,-\y)--(-3.45,-\y);
\foreach \y in {0,3,4,5}
\draw [thick, ->] (3.5,-\y)--(4.5,-\y);
\foreach \y in {0,1,3,5}
\draw [thick, ->] (-1.9,-\y)--(-1,-\y);
\foreach \y in {0,1,4,5}
\draw [thick,dashed, ->] (0.6,-\y)--(1.5,-\y);

\foreach \y in {1,3,4,5}
\draw [semithick,->] (5.4,-0.2-\y) arc(-87:-93:100);
\foreach \y in {1,3,4,5}
\draw (1,-0.5-\y) node {\tiny{$0$}};

\foreach \x in {0,1,2,3.16}
\draw [semithick, ->] (-4.5+2.5*\x,-5) arc(-10:10:13.8);
\foreach \x in {0,1,2,3.16}
\draw (-4.15+2.5*\x,-2) node {\tiny{$0$}};

\end{tikzpicture}
\end{center}
\end{fig}
Again, the crystal graph of $\B\ot \B^\vee$ can be defined locally by giving the vertices adjacent to the edges labelled $i$, as shown on Figure \ref{fig:crystali}.
\begin{fig}
\label{fig:crystali}
 \[\]
\begin{center}
\begin{tikzpicture}
%%crystal graph B
\draw (-2,0) node {$\B\ot\B^\vee(\xrightarrow[]{i})$ :};
\draw(-2,-1) node { \footnotesize{$k\notin \{i-1,i\}$}};
\draw (0,0) node[right]{$v_k\ot \vv_i$};
\draw (2,0) node[right] {{$v_{i-1}\ot \vv_i$}};
\draw (4.6,0) node[right] {$v_i \ot \vv_i$};
\draw (0,-1) node[right] {$v_k \ot \vv_{i-1}$};
\draw (4.6,-1) node[right] {$v_i \ot \vv_{i-1}$};
\draw (2,-2) node[right] {$v_{i-1}\ot \vv_k$};
\draw (4.6,-2) node[right] {$v_i\ot \vv_k$};
\draw [thick, ->] (0.68,-0.2)--(0.68,-0.8);
\draw (0.5,-0.5) node {\footnotesize{$i$}};
\draw [thick, ->] (5.23,-0.2)--(5.23,-0.8);
\draw (5.1,-0.5) node {\footnotesize{$i$}};
\draw [thick, ->] (3.6,0)--(4.6,0);
\draw (4.1,0.15) node {\footnotesize{$i$}};
\draw [thick, ->] (3.6,-2)--(4.6,-2);
\draw (4.1,-1.85) node {\footnotesize{$i$}};
\end{tikzpicture}
\end{center}
\end{fig}
We obtain, for all $i \in \{0, \dots, n-1\}$, the relations
\begin{equation}\label{eq:vexdual}
\begin{array}{lc}
\begin{cases}
\varphi_i (v_{i-1}\ot \vv_i) = \varepsilon_i (v_i\ot \vv_{i-1})= 2 \\
\varphi_i (v_{i}\ot \vv_{i-1}) = \varepsilon_i (v_{i-1}\ot \vv_{i})= 0 \\
\varphi_i (v_i\ot \vv_i)= \varepsilon_i (v_i\ot \vv_i)=1 \\
\varphi_i (v_{i-1}\ot \vv_{i-1})=\varepsilon_i (v_{i-1}\ot \vv_{i-1}) = 0
\end{cases},&\\\\
\begin{cases}
\varphi_i (v_{k}\ot \vv_i) =\varepsilon_i (v_i\ot \vv_k)= 1  \\
\varphi_i (v_{i-1}\ot \vv_k) =\varepsilon_i (v_k\ot \vv_{i-1})= 1 \\
\varphi_i (v_k \ot \vv_l)= \varepsilon_i (v_l \ot \vv_k) = 0 
\end{cases},&\forall\,\, l,k\notin\{i,i-1\}.
\end{array}
\end{equation}
The local configurations for the vertices are given in Figure \ref{fig:localconf}.

\begin{fig}\label{fig:localconf}
 \[\]
\begin{center}
\begin{tikzpicture}
%%crystal graph B
\draw(-6,-0.5) node {\footnotesize{$k-l\notin \{\pm 1\}$:}};
\draw (2,0) node[right] {{$v_{i-1}\ot \vv_i$}};
\draw (4.6,0) node[right] {$v_i \ot \vv_i$};
\draw (4.6,-1) node[right] {$v_i \ot \vv_{i-1}$};
\draw (-2.5,-0.5) node {$v_k \ot \vv_l$};

%% ii, i-1 i , ii_1
\draw [thick, ->] (5.23,-0.2)--(5.23,-0.8);
\draw (5.1,-0.5) node {\footnotesize{$i$}};

\draw [thick, ->] (5.23,-1.2)--(5.23,-1.8);
\draw (5.2,-1.5) node[left] {\footnotesize{$i-1$}};

\draw [thick, ->] (2.97,0.8)--(2.97,0.2);
\draw (2.95,0.5) node[left] {\footnotesize{$i+1$}};

\draw [thick, ->] (3.6,0)--(4.6,0);
\draw (4.1,0.15) node {\footnotesize{$i$}};

\draw [thick, ->] (6.2,-1)--(7.2,-1);
\draw (6.7,-0.85) node {\footnotesize{$i+1$}};

\draw [thick, ->] (0.9,0)--(1.9,0);
\draw (1.4,0.15) node {\footnotesize{$i-1$}};

%% k,l   
\draw [thick, ->] (-1.9,-0.5)--(-0.9,-0.5);
\draw (-1.4,-0.35) node {\footnotesize{$k+1$}};

\draw [thick, ->] (-4.2,-0.5)--(-3.2,-0.5);
\draw (-3.7,-0.35) node {\footnotesize{$k$}};

\draw [thick, ->] (-2.5,0.4)--(-2.5,-0.3);
\draw (-2.5,0.1) node[left] {\footnotesize{$l+1$}};

\draw [thick, ->] (-2.5,-0.7)--(-2.5,-1.5);
\draw (-2.5,-1.1) node[left] {\footnotesize{$l$}};
\end{tikzpicture}
\end{center}
\end{fig}

The values of the functions $\varepsilon,\varphi$ defined in \eqref{eq:epsilonphi} are
\begin{equation}\label{eq:pp}
\begin{array}{lc}
\begin{cases}
\varphi(v_{i-1}\ot \vv_i)= \varepsilon(v_i\ot \vv_{i-1})= 2\Lambda_i \\
\varepsilon (v_{i-1}\ot \vv_i)= \varphi (v_i\ot \vv_{i-1})= \Lambda_{i-1}+\Lambda_{i+1} \\
\varphi(v_i\ot \vv_i)= \varepsilon(v_i\ot \vv_i)= \Lambda_i \\
\end{cases},&\\\\
\begin{cases}
\varphi(v_k\ot \vv_l) = \Lambda_{k+1}+\Lambda_{l}\\
 \varepsilon(v_k\ot \vv_l) = \Lambda_{l+1}+\Lambda_{k}\\
\end{cases},
\end{array}
\end{equation}
where $k-l\notin\{0,\pm1\}.$ For all $k,l\in \I,$ the weight of $v_k\ot \vv_l$ is given by
\begin{equation}\label{eq:wtt}
 \wt (v_k\ot \vv_l) = \Lambda_{k+1}-\Lambda_{k} + \Lambda_{l}-\Lambda_{l+1}.
\end{equation}
We then observe that 
\begin{equation}\label{eq:lev}
\langle c, \varepsilon(v_k\ot \vv_l)\rangle = 1 + \chi(k\neq l).
\end{equation}

By \cite[Lemma 4.6.2]{KMN2a}, since $\B$ and $\B^\vee$ are perfect crystals of level $1$, their 
tensor product $\Bb$ is also a perfect crystal of level $1$.
We observe that the potential grounds of $\Bb$ are the vertices $ v_i\ot \vv_i$, since by \eqref{eq:pp}, for all $i\in \I$, we have that 
$$\varepsilon(b^{\Lambda_i }) = \Lambda_i \text{ if and only if } b^{\Lambda_i }= v_i\ot \vv_i\quad\text{and}\quad \varphi(b_{\Lambda_i }) = \Lambda_i \text{ if and only if } b_{\Lambda_i }= v_i\ot \vv_i.$$

\section{Proof of the character formulas}
\label{sec:proofchar}

In this section, we prove our character formulas given in   Theorems \ref{theorem:finalcharac}, \ref{theorem:finalcharacter}, and \ref{theorem:formulaexp}, under the assumption that \Thm{theorem:perfcrys} is true. We will then prove \Thm{theorem:perfcrys} in the last two sections.
\subsection{Proof of  \Thm{theorem:finalcharac}} 
We show that the set of grounded partitions $\Ppp$, with $\gg$ defined in \eqref{eq:relation}, grounded at $c_g$ for $g=(v_0\ot\vv_0)$, is in bijection with the set of generalised Primc partitions.

Let $(\pi_0,\ldots,\pi_{s-1},0_{\co}) \in \Ppp$, and let $b$ be the vertex in $\Bb$ corresponding to the colour of $\pi_{s-1}$.
Since $\pi_{s-1}\neq 0_{\co}$, the minimal size of $\pi_{s-1}$ is $H((v_0\ot\vv_0)\ot (v_i\ot\vv_j))$ if $(i,j)\neq (0,0)$, and $1$ if $(i,j)=(0,0).$ This corresponds to the size of minimal parts in generalised Primc partitions.
By \Thm{theorem:perfcrys}, for all $(i,j)\neq (0,0),$
$$H((v_0\ot\vv_0)\ot (v_i\ot\vv_j))=\Delta(a_jb_i,a_0b_0)=1.$$
Thus the generalised Primc partitions and the grounded partitions in $\Ppp$ coincide in terms of minimal difference conditions 
and minimal part sizes,
with the colour correspondence $c_{v_l\ot\vv_k}\leftrightarrow a_kb_l$. Thus their generating functions are the same with the correspondence
$e^{\wt v_i} = b_i$, % and $a_i=b_i^{-1}$,
since by \eqref{eq:wtt},
$$e^{\wt(v_l\ot\vv_k)} = e^{\wt(l)-\wt(k)} = b_k^{-1}b_l.$$ % = a_kb_l\,\cdot$$
Using the character formula of \Thm{theorem:formchar}, this gives the desired result. \qed

\subsection{Proof of \Thm{theorem:finalcharacter}}
Let us now turn to the proof of \Thm{theorem:finalcharacter}. It uses some notions defined in our first paper \cite{DK19}, such as bound and free colours, reduced colour sequences, kernel, insertions, types. As they are only needed for this proof, we do not redefine them here, and refer the reader to Sections 1 and 2 of \cite{DK19}.

Let us fix $\ell\in\{0,\ldots,n-1\}$ and recall that in the perfect crystal $\Bb$, we have $b^{\Lambda_{\ell} }= b_{\Lambda_{\ell} } = v_{\ell}\ot \vv_{\ell}$. Assuming that \Thm{theorem:perfcrys} is true, we also have that $H[(v_{\ell}\ot \vv_{\ell})\ot (v_{\ell}\ot \vv_{\ell})] = \Delta(a_{\ell}b_{\ell};a_{\ell }b_{\ell})=0$. 
Let us set $g=(v_{\ell}\ot\vv_{\ell})$ to be the ground in $\Bb$, and consider the set $\Ppp$ of grounded partitions with ground $\co$. For $\pi = (\pi_0,\ldots,\pi_{s-1}, 0_{\co})\in \Ppp$, we write  
$c(\pi_k)= c_{(v_{j_k}\ot \vv_{i_k})}.$

By \eqref{eq:relation}, for all $k\in\{0,\cdots,s-2\},$ the parts of $\pi$ satisfy the difference conditions
\[\pi_k-\pi_{k+1}\geq  H((v_{j_{k+1}}\ot \vv_{i_{k+1}})\ot (v_{j_k}\ot \vv_{i_k}))=\Delta(a_{i_k}b_{j_k},a_{i_{k+1}}b_{j_{k+1}}).\]

Let us now study the size of the smallest part.
\begin{itemize}
\item If $(j_{s-1},i_{s-1})\neq (\ell,\ell)$, the minimal size of $\pi_{s-1}$ is 
\begin{equation*}
 \Delta(a_{i_{s-1}}b_{j_{s-1}},a_{\ell} b_{\ell})= 
 \begin{cases}  
 \chi(i_{s-1}>\ell) = \chi(i_{s-1}\geq \ell)+\chi(\ell>j_{s-1}) &\text{ if}\quad j_{s-1}=\ell \quad(i_{s-1}\neq \ell)\\
\chi(i_{s-1}\geq \ell)+\chi(\ell>j_{s-1}) &\text{ if}\quad j_{s-1}\neq \ell
  \end{cases}.
\end{equation*}
\item Otherwise, we have that $j_{s-1}=i_{s-1}=\ell$, and then $c(\pi_{s-1})=c_g$. In this case, $\Delta(a_{\ell}b_{\ell};a_{\ell }b_{\ell})=0$ implies that the size of $\pi_{s-1}$ must be at least $1$ in order to have $\pi_{s-1}\neq 0_{\co}.$
We observe that, in that case, we still have $1 = \chi(i_{s-1}\geq \ell)+\chi(\ell>j_{s-1})$.
\end{itemize}
In both cases, our grounded partition $\pi$, without the part $0_{c_g}$, is a partition satisfying the difference condition $\Delta$ of generalised Primc partitions, but such that the minimal size for the last part, denoted by $\Delta(a_{i_{s-1}}b_{j_{s-1}},a_\infty b_\infty)$ with our conventions from \cite{DK19}, is given by the expression
\begin{equation}
\label{eq:minpart}
\Delta(a_{i_{s-1}}b_{j_{s-1}},a_\infty b_\infty) = \chi(i_{s-1}\geq \ell)+\chi(\ell>j_{s-1}).
\end{equation}
Moreover we observe that this is always equal to $1$ when $a_{i_{s-1}}b_{j_{s-1}}$ is a free colour. 
Thus in the case $\ell=0$, the minimal part always has size $1$, independently of its colour. For larger $\ell$, the minimal part may have size $0,1,$ or $2$ according to \eqref{eq:minpart}.
Besides, we keep the convention 
$\Delta(a_\infty b_\infty,c)=1,$ as it our first paper.

\medskip
The proof of \Thm{theorem:main2}  in \cite{DK19} relies on a correspondence between generalised Primc partitions and coloured Frobenius partitions having the same kernel. In the case where the kernel ends with a free colour $a_kb_k$, the generalised Primc partition is also a partition grounded in $c_g$ by adding $0_{c_g}$, and the type of the insertions inside the secondary pairs remain the same. 

When the kernel ends with a bound colour  $a_kb_{k'}$, $k \neq k'$, we modify the type of the insertion of $a_{k'}b_{k'}$ to the right of $a_kb_{k'}$, and it becomes
\begin{align}\label{eq:inst1}
T_{\Delta}(a_kb_{k'})&=\Delta(a_kb_{k'},a_{k'}b_{k'})+\Delta(a_kb_{k'},a_\infty b_\infty)-\Delta(a_kb_{k'},a_\infty b_\infty)\nonumber\\
&= 1+\chi(k>k')-(\chi(k\geq \ell)+\chi(\ell>k')).
\end{align}
Note that this value is still in $\{0,1\}$, since it can be rewritten as $\chi(\ell> k)+\chi(k>k')-\chi(\ell>k').$ The types of the others insertions are the same as the types for the generalised Primc partitions in \cite{DK19}.

\medskip
Recall from \cite{DK19} that a \textit{$n^2$-coloured Frobenius partition} is a pair of coloured partitions
$$\begin{pmatrix}
\lambda_0 & \lambda_1 & \cdots & \lambda_{s-1} \\
\mu_0 & \mu_1 & \cdots & \mu_{s-1} 
\end{pmatrix},$$
where $\lambda =\lambda_0 + \lambda_1 + \cdots + \lambda_{s-1}$  is a partition into $s$ distinct non-negative parts, each coloured with some $a_i$, $i \in \{0, \dots, n-1\},$ with the following order
\begin{equation}
\label{eq:orderFroba}
0_{a_{n-1}} < 0_{a_{n-2}} < \cdots < 0_{a_0} < 1_{a_{n-1}} < 1_{a_{n-2}} < \cdots < 1_{a_0} < \cdots,
\end{equation}
and $\mu=\mu_0+\mu_1+\cdots+\mu_{s-1}$ is a partition into $s$ distinct non-negative parts, each coloured with some $b_i$, $i \in \{0, \dots, n-1\},$ with the order
\begin{equation}
\label{eq:orderFrobb}
0_{b_{0}} < 0_{b_1} < \cdots < 0_{b_{n-1}} <1_{b_{0}} < 1_{b_1} < \cdots < 1_{b_{n-1}} < \cdots .
\end{equation}
The colour sequence of such a partition is defined to be $c(\lambda_0)c(\mu_0), \dots ,c(\lambda_{s-1})c(\mu_{s-1})$. Here the size corresponding to the colour $c(\lambda_i)c(\mu_i)$ is $\lambda_i+\mu_i$.

We consider coloured Frobenius partitions such that the minimal size for $\lambda_{s-1}+\mu_{s-1}$ is given by $\Delta'(a_kb_{k'},a_\infty b_\infty)=\Delta(a_kb_{k'},a_\infty b_\infty)$, where
 $c(\lambda_{s-1})= a_k,$ $c(\mu_{s-1})=b_{k'}$, and
 $\Delta(a_kb_{k'},a_\infty b_\infty)$ was defined in \eqref{eq:minpart}. 
We say that such coloured Frobenius partitions are \textit{grounded} at $c_g$.
We have $\Delta'(a_kb_{k},a_\infty b_\infty)=1$ for any free colour $a_kb_k.$
Note that the differences are the same as those defined in \cite{DK19}:
\begin{equation*}
 \Delta'(a_ib_{j},a_{i'}b_{j'}) = \chi(i\geq i')+\chi(j\leq j').
\end{equation*}
Here we keep the convention $\Delta'(a_\infty b_\infty,c)=1$.
When the kernel of the coloured Frobenius partition ends with a bound colour $a_{k}b_{k'}$, the type of the insertion of the free colour $a_{k'}b_{k'}$ to its right, according to the differences $\Delta'':=2-\Delta'$, is given by  
\begin{align}\label{eq:inst2}
T_{\Delta''}(a_kb_{k'})&= \Delta''(a_kb_{k'},a_{k'}b_{k'}) + \Delta''(a_{k'}b_{k'},a_\infty b_\infty)-\Delta''(a_kb_{k'},a_\infty b_\infty)\nonumber\\
&= 2-[\Delta'(a_kb_{k'},a_{k'}b_{k'}) + \Delta'(a_{k'}b_{k'},a_\infty b_\infty)-\Delta'(a_kb_{k'},a_\infty b_\infty)]\nonumber\\
&= 2 -[1+\chi(k>k')+1 - (\chi(k\geq \ell)+\chi(\ell>k'))]\nonumber\\
&= \chi(k\geq \ell)+\chi(\ell>k')-\chi(k> k'). 
\end{align}

The types of all the insertions which are not at the right end of the kernel are the same as the types for $\Delta''$ of the coloured Frobenius partitions in \cite{DK19}.
Thus, \eqref{eq:inst1} yields the relation 
\begin{equation*}
T_{\Delta}(a_kb_{k'})+T_{\Delta''}(a_kb_{k'}) =1.
\end{equation*}
This means that an insertion has $\Delta$-type $1$ if and only if it has  $\Delta''$-type $0$. Thus, by Theorem 3.1 of \cite{DK19}, the generating function for our \textit{grounded} generalised Primc partitions with a fixed kernel is the same as  the generating function for grounded coloured Frobenius partitions with the same kernel. Therefore, the generating function for generalised Primc partitions with minimal part size $\Delta(a_kb_{k'},a_\infty b_\infty)$ is the same 
as the generating function for coloured Frobenius partitions with  minimal part size $\Delta'(a_kb_{k'},a_\infty b_\infty)=\chi(k\geq \ell)+\chi(\ell>k')$. The generating function for the latter, where for all $i \in\I,$ the power of $b_i$ counts the number of colours $b_i$ minus the number of colours $a_i$ in the colour sequence, is given by 
$$[x^0]\prod_{i=0}^{\ell-1}(-b_i^{-1}x;q)_{\infty}(-b_iqx^{-1};q)_{\infty}\times\prod_{i=\ell}^{n-1}(-b_i^{-1}xq;q)_{\infty}(-b_ix^{-1};q)_{\infty}.$$

In this product, the minimal size for $\lambda_{s-1}$ with colour $a_k$ is $\chi(k\geq \ell)$, while the minimal size for $\mu_{s-1}$ with colour $b_{k'}$ is $\chi(k'<\ell)$,
so that the minimal size for $\lambda_{s-1}+\mu_{s-1}$ is indeed $\chi(k\geq \ell)+\chi(\ell>k')$. We conclude by noting that, by \Thm{theorem:main2}, this generating function is obtained by 
doing the changes of variables $b_i \mapsto b_iq^{\chi(i<\ell)}$ in 
\[G^P_{n}(q;b_0,\cdots,b_{n-1})=[x^0]\prod_{i=0}^{n-1}(-b^{-1}_ixq;q)_{\infty}(-b_ix^{-1};q)_{\infty},\]
which gives \Thm{theorem:finalcharacter}. \qed

\subsection{Proof of  \Thm{theorem:formulaexp}}
Finally, we turn to the proof of \Thm{theorem:formulaexp}, which gives the expression of the character for $L(\Lambda_{\ell})$ as a sum of series with positive coefficients.

By the definition of characters, the function $e^{-\Lambda_{\ell}}\rm{ch}(L(\Lambda_{\ell}))$ can be expressed as a power series in $e^{-\alpha_i}$ for $i\in \I$, or, by a change of variables, as a power series in $e^{-\delta}$ and $e^{\alpha_i}$ for $i\neq 0$.
By definition of the crystal graph $\B$, we have $\tilde f_i v_{i-1}= v_i$, so that by \eqref{eq:weightalpha}, we have $\wt v_{i-1}-\wt v_i=\alpha_i$ for $i \in \{1, \dots,n-1\}$ and $\wt v_{n-1}-\wt v_0=\ov \alpha_0.$ The change of variables $e^{\wt v_i}=b_i$ then gives  $e^{\alpha_i}= b_{i-1}b^{-1}_{i}$ for $i \in \{1, \dots,n-1\}$.
%Recall that on the crystal $\B$, $\alpha_i$ is viewed as its restriction on $\bar P,$ so that $\sum_{i\in I} \alpha_i = 0$ (since $\delta$ is the null root and vanishes on $\bar \h$).
%It is thus  coherent to have 
%$$e^{\ov \alpha_0} = b_{n-1}b^{-1}_0 = \prod_{i=1}^{n-1} b_{i}b_{i-1}^{-1} = \prod_{i=1}^{n-1}e^{-\alpha_i}.$$
The changes of variables are natural, since for all $i\neq 0$, the weight $\alpha_i$ in $P$ is indeed a classical weight in $\bar P$. 
In addition, the series $G_n^P(b_0q,\cdots,b_{\ell-1}q,b_{\ell},\cdots,b_{n-1})$  can be expressed in terms of summands  
of the form 
$$\left(\prod_{i=0}^{n-1}b_i^{r_i}\right)q^m\quad\text{with} \quad \sum_{i=0}^{n-1} r_i =0,$$
so that we can always retrieve the exponent of  $b_{i-1}b^{-1}_i$, for all $i\in \{1,\ldots,n-1\}$, which corresponds to $\sum_{j=0}^{i-1} r_j.$ Thus the identification 
\begin{align*}
e^{-\delta} &\longleftrightarrow q \\
e^{\alpha_i} &\longleftrightarrow b_{i-1}b_i^{-1}
\end{align*}
is unique, and our generalisation of Primc's identity allows us to retrieve the non-dilated version of the characters 
for all the irreducible highest weight modules with classical weight of level 1 for the type $A_{n-1}^{(1)}$. 

\medskip
Looking at Formula  \eqref{eq:formulefinale}, we obtain the following correspondences (recall that $r_1=0=r_n$)
\begin{align*}
 \prod_{i=1}^{n-1} b_i^{-r_i+r_{i+1}} &= \prod_{i=1}^{n-1} (b_{i-1}b_i^{-1})^{r_i} = \prod_{i=1}^{n-1}e^{r_i\alpha_i},\\
 \prod_{j=0}^{i-1} b_{j}b_i^{-1}&= \prod_{j=1}^{i} (b_{j-1}b_j^{-1})^j= e^{\sum_{j=1}^i j\alpha_j}.\\
\end{align*}
By doing
these transformations in \eqref{eq:formulefinale}, we then obtain by \Thm{theorem:finalcharac} and Theorem \ref{theorem:main2} that 
\begin{align*}
e^{-\Lambda_0}\mathrm{ch}(L(\Lambda_0))&=\frac{1}{(e^{-\delta};e^{-\delta})_{\infty}^{n-1}}\sum_{\substack{s_1, \dots, s_{n-1}\in \Z\\s_n=0}}\prod_{i=1}^{n-1} e^{s_i\alpha_i}e^{s_i(s_{i+1}-s_{i})\delta}
\\&=\frac{1}{(e^{-\delta};e^{-\delta})_{\infty}^{n-1}} \sum_{\substack{r_1, \dots, r_{n-1}:\\0 \leq r_j \leq j-1\\r_n=0}} \prod_{i=1}^{n-1} e^{r_i\alpha_i}e^{r_i(r_{i+1}-r_i)\delta} \left(e^{-i(i+1)\delta};e^{-i(i+1)\delta}\right)_{\infty}
\\& \qquad \qquad \qquad \times \left(-e^{(ir_{i+1}-(i+1)r_i-\frac{i(i+1)}{2})\delta+\sum_{j=1}^i j\alpha_j};e^{-i(i+1)\delta}\right)_{\infty}
\\& \qquad \qquad \qquad \times \left(-e^{((i+1)r_i-ir_{i+1}-\frac{i(i+1)}{2})\delta-\sum_{j=1}^i j\alpha_j};e^{-i(i+1)\delta}\right)_{\infty}.
\end{align*}
Note that for all $\ell\in \{0,\ldots,n-1\}$ and $j\in \{1,\ldots,n-1\}$, the transformation $b_j \mapsto b_jq^{\chi(j<\ell)}$ is equivalent to $b_{j-1}b_j^{-1}\mapsto q^{\chi(j=\ell)}b_{j-1}b_j^{-1}.$
This corresponds to the transformations $e^{\alpha_j}\mapsto e^{-\chi(j=\ell)\delta+\alpha_j}$ for all $j \in \{1, \dots, n-1\},$ and \Thm{theorem:formulaexp} follows. \qed

\section{Tools for the proof of \Thm{theorem:perfcrys}}
\label{sec:toolperfcrys}
We already know that the crystal graph of $\Bb\ot\Bb$ is connected, as $\Bb$ is a perfect crystal. However, here we reprove this by building the paths in this graph, as these paths will allow us to compute the energy function.
First, let us define some tools that will help us simplify the construction of the paths.

\subsection{Symmetry in the crystal graph of $\Bb\ot \Bb$}
We use some properties of duality to exhibit symmetries in the crystal graph of $\Bb\ot\Bb$. Let us start with a general proposition about energy and duality.

\begin{prop}
\label{prop.energy_dual}
Let $\B$ be a self-dual crystal and $H$  be an energy function on $\B \ot \B$. Then for all $b_1, b_2 \in \B$, we have
\begin{equation}\label{eq:same_diff_energy}
H[b_1\ot b_2] -H[\tilde f_i(b_1\ot b_2)] = H[b_2^{\vee} \ot b_1^\vee]-H[\tilde e_i(b_2^\vee \ot b_1^\vee)].
\end{equation}
\end{prop}
\begin{proof}
By \eqref{eq:enfct}, we have
$$\begin{array}{c}
 H[b_1 \ot b_2] - H[\tilde f_i (b_1 \ot b_2)]
= \begin{cases}
- \chi(i=0) & \qquad \hbox{\rm if} \ \ \varphi_i(b_1) \leq \varepsilon_i(b_2)  \\
+ \chi(i=0) & \qquad \hbox{\rm if} \ \ \varphi_i(b_1) > \varepsilon_i(b_2),
\end{cases} \\
H[b_2^\vee \ot b_1^\vee] - H[\tilde e_i (b_2^\vee \ot b_1^\vee)]
= \begin{cases}
-\chi(i=0) & \qquad \hbox{\rm if} \ \ \varphi_i(b_2^\vee)  \geq \varepsilon_i(b_1^\vee)  \\
+ \chi(i=0) & \qquad \hbox{\rm if} \ \ \varphi_i(b_2^\vee)< \varepsilon_i(b_1^\vee).
\end{cases}
\end{array}$$
On the other hand, we have by duality :
\begin{align*}
\varphi_i(b_2^\vee) &=\varepsilon_i(b_2),\\
\varepsilon_i(b_1^\vee) &= \varphi_i(b_1).
\end{align*}
The equality \eqref{eq:same_diff_energy} follows.
\end{proof}

We now use Proposition \ref{prop.energy_dual} to give the symmetry in $\Bb\ot\Bb.$

\begin{prop}\label{prop:pathreduc}
Let $\B$ be a crystal, let $\B^\vee$ be its dual, and let $\Bb =\B\ot\B^\vee$. Let $H$ be an energy function on $\Bb\ot\Bb.$
%Denote by $\sigma^\vee$ the element in $\B^\vee$ corresponding to $\sigma\in\B$. Then
%for any $\sigma_1,\sigma_2,\sigma_3,\sigma_4,\tau_1,\tau_2,\tau_3,\tau_4\in \B$, 
%we have the following equivalence in the crystal $\Bb\ot\Bb$ :
%\begin{equation}\label{eq:permut}
%\tilde f_i [(\sigma_1\ot\sv_2)\ot (\sigma_3\ot\sv_4)]=(\tau_1\ot\tau_2^{\vee})\ot (\tau_3\ot\tau_4^\vee)
%\Longleftrightarrow \tilde e_i [(\sigma_4\ot\sv_3)\ot (\sigma_2\ot\sv_1)] = (\tau_4\ot\tau_3^{\vee})\ot (\tau_2\ot\tau_1^\vee),
%\end{equation}
%and an energy function $H$ on $\Bb\ot\Bb$ satisfies 
%\begin{equation}\label{eq:memediff}
%H[(\sigma_1\ot\sv_2)\ot (\sigma_3\ot\sv_4)]-H[\tilde f_i((\sigma_1\ot\sv_2)\ot (\sigma_3\ot\sv_4))] = H[(\sigma_4\ot\sv_3)\ot (\sigma_2\ot\sv_1)]-H[\tilde e_i((\sigma_4\ot\sv_3)\ot (\sigma_2\ot\sv_1))].
%\end{equation}

For all $\sigma_1,\sigma_2,\sigma_3,\sigma_4,\tau_1,\tau_2,\tau_3,\tau_4\in \B$, there exists a path between $(\sigma_1\ot\sv_2)\ot (\sigma_3\ot\sv_4)$ and $(\tau_1\ot\tau_2^{\vee})\ot (\tau_3\ot\tau_4^\vee)$ in $\Bb \ot \Bb$ if and only if there exists a path between 
$(\sigma_4\ot\sv_3)\ot (\sigma_2\ot\sv_1)$ and $(\tau_4\ot\tau_3^{\vee})\ot (\tau_2\ot\tau_1^\vee)$ in $\Bb \ot \Bb$. These paths are dual, i.e. we can obtain one from the other by reversing the edges and replacing each vertex by its dual.

Moreover, if there exists a path between $(\sigma_1\ot\sv_2)\ot (\sigma_3\ot\sv_4)$ and $(\tau_1\ot\tau_2^{\vee})\ot (\tau_2\ot\tau_1^\vee)$, we have
\begin{equation}\label{eq:egalite}
 H[(\sigma_1\ot\sv_2)\ot (\sigma_3\ot\sv_4)] = H[(\sigma_4\ot\sv_3)\ot (\sigma_2\ot\sv_1)].
\end{equation}
\end{prop}

\begin{proof}
Note that $\Bb \ot \Bb$ is self-dual.

The first claim about the paths follows directly from the definition of duality.

We prove \eqref{eq:egalite} by induction on the length of the path between $(\sigma_1\ot\sv_2)\ot (\sigma_3\ot\sv_4)$ and $(\tau_1\ot\tau_2^{\vee})\ot (\tau_2\ot\tau_1^\vee)$.
\begin{itemize}
\item If the path is of length $0$, i.e. we consider the vertex $(\tau_1\ot\tau_2^{\vee})\ot (\tau_2\ot\tau_1^\vee)$, \eqref{eq:egalite} is simply the identity.

\item Now assume that \eqref{eq:egalite} is true for all vertices $(\sigma_1\ot\sv_2)\ot (\sigma_3\ot\sv_4)$ having a path of length $n$ to $(\tau_1\ot\tau_2^{\vee})\ot (\tau_2\ot\tau_1^\vee)$, and show it for paths of length $n+1$.
Let $(\sigma_1\ot\sv_2)\ot (\sigma_3\ot\sv_4)$ be such that there is a path of length $n+1$ between $(\sigma_1\ot\sv_2)\ot (\sigma_3\ot\sv_4)$ and $(\tau_1\ot\tau_2^{\vee})\ot (\tau_2\ot\tau_1^\vee)$.

Then there exists $i$ such that 
$$(\sigma_1\ot\sv_2)\ot (\sigma_3\ot\sv_4) = g_i \big((\rho_1\ot\rho^\vee_2)\ot (\rho_3\ot\rho^\vee_4)\big),$$
where $g_i$ is equal to either $\tilde f_i$ or $\tilde e_i$, and such that there is a path of length $n$ between $(\rho_1\ot\rho^\vee_2)\ot (\rho_3\ot\rho^\vee_4)$ and  $(\tau_1\ot\tau_2^{\vee})\ot (\tau_2\ot\tau_1^\vee)$.

By the induction hypothesis, we have 
$$ H[(\rho_1\ot\rho^\vee_2)\ot (\rho_3\ot\rho^\vee_4)] = H[(\rho_4\ot\rho^\vee_3)\ot (\rho_2\ot\rho^\vee_1)].$$
Moreover, by Proposition \ref{prop.energy_dual}, we have
$$ H[(\sigma_1\ot\sv_2)\ot (\sigma_3\ot\sv_4)] - H[(\rho_1\ot\rho^\vee_2)\ot (\rho_3\ot\rho^\vee_4)] = H[(\sigma_4\ot\sv_3)\ot (\sigma_2\ot\sv_1)] - H[(\rho_4\ot\rho^\vee_3)\ot (\rho_2\ot\rho^\vee_1)].$$
Combining the two equalities completes the proof.
\end{itemize}
\end{proof}

In particular, by Proposition \ref{prop:pathreduc}, if we find a path from 
$(v_{0}\ot\vv_{0})\ot (v_0\ot\vv_0)$ to $(v_{l'}\ot\vv_{k'})\ot (v_l\ot\vv_k)$, then we immediately have a path from
$(v_{0}\ot\vv_{0})\ot (v_0\ot\vv_0)$ to $(v_{k}\ot\vv_{l})\ot (v_{k'}\ot\vv_{l'})$ as well, and by \eqref{eq:egalite}, this yields the following symmetry on the energy function:
\[
 H[(v_{l'}\ot\vv_{k'})\ot (v_l\ot\vv_k)]=H[(v_{k}\ot\vv_{l})\ot (v_{k'}\ot\vv_{l'})].
\]
Besides,  by \eqref{eq:Delta}, we have
\begin{align}
\Delta(a_kb_l;a_{k'}b_{l'}) & = \chi(k\geq k')-\chi(k=l=k')+\chi(l\leq l')-\chi(l=k'=l')\nonumber\\
& = 
\begin{cases}
\chi(k>k')+ \chi(l<l')\quad\text{if}\quad l=k'\\
\chi(k\geq k')+ \chi(l\leq l')\quad\text{if}\quad l\neq k'
\end{cases},\label{eq:primcreducdiff}
\end{align}
and then 
\[ \Delta(a_kb_l;a_{k'}b_{l'}) = \Delta(a_{l'}b_{k'};a_{l}b_{k}).\]
Therefore, if we prove that  $H[(v_{l'}\ot\vv_{k'})\ot (v_l\ot\vv_k)]= \Delta(a_kb_l;a_{k'}b_{l'})$, we equivalently have  
$H[(v_{k}\ot\vv_{l})\ot (v_{k'}\ot\vv_{l'})] = \Delta(a_{l'}b_{k'};a_{l}b_{k})$. Thus, to prove Theorem \ref{theorem:perfcrys} in Section \ref{sec:proofcrys}, we will distinguish several cases according to some relations between $k,k',l,l'$, and by interchanging $k\equiv l'$ and $k'\equiv l$, the symmetry will imply the remaining cases.

\subsection{Redefining the minimal differences $\Delta$}
To build a path from $(v_{0}\ot\vv_{0})\ot (v_0\ot\vv_0)$ to $(v_{l'}\ot\vv_{k'})\ot (v_l\ot\vv_k)$ and show that 
\[H[(v_{l'}\ot\vv_{k'})\ot (v_l\ot\vv_k)]=\Delta(a_kb_l;a_{k'}b_{l'}),\]
we will distinguish the cases  $k'=l$ and $k'\neq l$. But first, let us define a tool which will make our proofs simpler.
\begin{defn}\label{defn:tool}
Let us identify $\I$ with $\Z/n\Z$, and consider the natural order on $\I$,
\[0<1<\cdots<n-2<n-1.\]
We also define, for all $i,j\in \I,$ the intervals 
\[\inter(i,j):=\{i+1,i+2,\ldots,j-1,j\}.\]
\end{defn}
\begin{lem}\label{lem:diffint}
For all $i \in \I$, we have the following:
\begin{equation*}
\begin{array}{ccc}
i<i-1&\Longleftrightarrow& i=0,\\
\inter(i,i)& = &\I,\\ 
\I\setminus \inter(i,j)= \inter (j,i) &\Longleftrightarrow& i\neq j,\\
0\notin \inter(j,i)&\Longleftrightarrow&j<i,\\
0\in \inter(i,j)&\Longleftrightarrow& j\leq i.\\
\end{array}
\end{equation*}
\end{lem}
The aim of this lemma is to rewrite the difference conditions $\Delta$ according to the fact that $0$ belongs to some interval or not.
By \eqref{eq:primcreducdiff}, $\Delta$ can be reformulated as follows:
\begin{equation}\label{eq:finfin}
\Delta(a_kb_l;a_{k'}b_{l'})  = 
\begin{cases}
\chi(0\notin \inter(k',k))+\chi(0\notin \inter(l,l'))\quad\text{if}\quad l=k'\\
\chi(0\in \inter(k,k'))+\chi(0\in \inter(l',l))\quad\text{if}\quad l\neq k'
\end{cases}.
\end{equation}

\begin{proof}[Proof of Lemma \ref{lem:diffint}]
 The first equivalence is straightforward, since $i>i-1$ if and only if $i\neq 0$, and $0<n-1=-1$.
 
 The second equality follows from the definition of $\inter$, since we go around $\I$.

 Note that 
 \[\inter(i,j) = \{i+1,i+2,\ldots,j-1,j\},\]
while 
$$\inter(j,i) = \{j+1,j+2,\ldots,i-1,i\},$$
 and if $i\neq j$, these two sets are complementary in $\I$. Moreover, when $i \neq j,$
 we have $i\in \inter(j,i)$ and $j\in \inter(i,j)$, so that both sets never equal $\emptyset$ nor $\I$. Otherwise, when $i=j$, they would both be equal to $\I$. This gives the third equivalence.

 For the fourth equivalence, the fact that $0\in \I$ gives  
 \begin{align*}
  0\notin \inter(j,i) &\Longleftrightarrow 0\notin \{j+1,i+2,\ldots,j-1,i\},\\
 &\Longleftrightarrow  i\neq j \text{ and } \emptyset \neq \{j+1,j+2,\ldots,i-1,i\}\subseteq \{1,\ldots,n-1\}\\
 &\Longleftrightarrow j<j+1\leq i.
 \end{align*}
 Finally, for the last equivalence, we note that 
 \begin{align*}
  \chi(j\leq i) &= \chi(j<i)+\chi(j=i)\\
   &= \chi(j<i)\chi(j\neq i)+\chi(j=i)\\
  &=\chi( 0\notin \inter(j,i))\chi(i\neq j) +\chi(i=j)\\
  &=\chi( 0\in \inter(i,j))\chi(i\neq j) +\chi(i=j)\chi(0\in \inter(i,i))  .
 \end{align*}
This concludes the proof.
\end{proof}

\section{Proof of \Thm{theorem:perfcrys}}
\label{sec:proofcrys}
We are now ready to build the paths in $\Bb\ot \Bb$, and use them to compute the energy function
$H[(v_{l'}\ot\vv_{k'})\ot (v_l\ot\vv_k)].$ We will use the relations in \eqref{eq:vexdual} and the local configurations of the vertices as defined in \eqref{fig:localconf}. 
The symmetric of $(v_{l'}\ot\vv_{k'})\ot (v_l\ot\vv_k)$ is $(v_{k}\ot\vv_{l})\ot (v_{k'}\ot\vv_{l'})$, obtained by interchanging $k'\equiv l, l'\equiv k$.
We distinguish several cases:
\begin{enumerate}
 \item $k'=l'$ and $l=k$,
 \item $k'=l\neq k=l'$,
 \item $k'=l$ and $k\neq l'$,
 \item $k'\neq k=l=l'$ (Symmetric: $l\neq k=k'=l'$),
 \item $l'\neq k'=k\neq l$ (Symmetric: $k\neq l=l' \neq k'$),
 \item $k\neq k'$, $k'\neq l$ and $l\neq l'$
 \begin{enumerate}
  \item $k+1,k'\notin \inter(l,l')$ (Symmetric: $l'+1 ,l\notin \inter(k',k)$),
  \item $k+1\in int(l,l')$ and $k'\notin \inter(l,l')$ (Symmetric: $l'+1 \in \inter(k',k)$ and $l\notin \inter(k',k)$)
  \item $k+1\notin \inter(l,l')$ and $k'\in \inter(l,l')$ (Symmetric: $l'+1 \notin \inter(k',k)$ and $l\in \inter(k',k)$)
  \item $k+1,k'\in \inter(l,l')$ and $l'+1,l\in \inter(k',k)$.
   \end{enumerate}
\end{enumerate}

\subsection{The case $k'=l'$ and $l=k$}
We want to compute the energy $H[(v_{k'}\ot\vv_{k'})\ot (v_l\ot\vv_l)]$ for all $k',l$. To do so, we build a path from $(v_{k'}\ot\vv_{k'})\ot (v_{k'}\ot\vv_{k'})$ to $(v_{k'}\ot\vv_{k'})\ot (v_l\ot\vv_l)$. We consider the case $k'\neq l$, as otherwise the two elements are the same.
By \eqref{eq:vexdual}, we have
 $$\varphi_i(v_{k'}\ot\vv_{k'}) = \varepsilon_i(v_{k'}\ot\vv_{k'}) = \chi(i=k').$$
By the tensor rules \eqref{eq:tensrul}, we then obtain the path
\begin{align*}
(v_{k'}\ot\vv_{k'})\ot (v_{k'}\ot\vv_{k'})\xrightarrow[]{\,\,\,k'\,\,}\underbrace{(v_{k'}\ot\vv_{k'})\ot (v_{k'}\ot\vv_{k'-1})\xrightarrow[]{k'-1}\cdots \xrightarrow[]{\,l+1}}_{\text{empty if }k'=l+1}(v_{k'}\ot\vv_{k'})\ot (v_{k'}\ot\vv_{l})\\
\Big\downarrow {\scriptstyle k'+1}\qquad\qquad \\
(v_{k'}\ot\vv_{k'})\ot (v_{l}\ot\vv_{l})\xleftarrow[]{\,\,\,l\,\,\,\,}\underbrace{(v_{k'}\ot\vv_{k'})\ot (v_{l-1}\ot\vv_l)\xleftarrow[]{l-1}\cdots \xleftarrow[]{k'+2}(v_{k'}\ot\vv_{k'})\ot (v_{k'+1}\ot\vv_{l})}_{\text{empty if }k'+1=l}\\
\end{align*}
This path is only made of forward moves $\tilde f_i$, with $i\in \inter(l,k')\sqcup \inter(k',l)$ appearing once, where we change the right side of the tensor products. By \eqref{eq:enfct}, we then have 
\begin{equation}
\label{eq:H1}
H[(v_{k'}\ot\vv_{k'})\ot (v_{l}\ot\vv_{l})]-H[(v_{k'}\ot\vv_{k'})\ot (v_{k'}\ot\vv_{k'})] = \chi(0\in \inter(l,k'))+\chi(0\in \inter(k',l))=1.
\end{equation}
By Proposition \ref{prop:pathreduc}, we have the dual path from $(v_{k'}\ot\vv_{k'})\ot (v_{k'}\ot\vv_{k'})$ to $(v_{l}\ot\vv_{l})\ot (v_{k'}\ot\vv_{k'})$, and
\begin{equation}
\label{eq:Hsym1}
H[(v_{l}\ot\vv_{l})\ot (v_{k'}\ot\vv_{k'})]-H[(v_{k'}\ot\vv_{k'})\ot (v_{k'}\ot\vv_{k'})]=1.
\end{equation}
Here we need to compute $H((v_{k'}\ot\vv_{k'})\ot (v_{k'}\ot\vv_{k'})).$ By interchanging $k'$ and $l$, we obtain
\begin{equation}
\label{eq:H1'}
H[(v_{k'}\ot\vv_{k'})\ot (v_{l}\ot\vv_{l})]-H[(v_{l}\ot\vv_{l})\ot (v_{l}\ot\vv_{l})]= 1.
\end{equation}
Subtracting \eqref{eq:H1} to \eqref{eq:H1'} yields
\[H[(v_{k'}\ot\vv_{k'})\ot (v_{k'}\ot\vv_{k'})]=H[(v_{l}\ot\vv_{l})\ot (v_{l}\ot\vv_{l})],\]
and we have an explicit path from $(v_{l}\ot\vv_{l})\ot (v_{l}\ot\vv_{l})$ to $(v_{k'}\ot\vv_{k'})\ot (v_{k'}\ot\vv_{k'})$ by combining the previous paths.

Recall that by definition, $ H[(v_{0}\ot\vv_{0})\ot (v_{0}\ot\vv_{0})]=0.$ Thus setting $k'=0$ yields by \eqref{eq:finfin} that for all $l\in \I$,
\begin{equation}\label{eq:premcas}
\begin{aligned}
 H[(v_{l}\ot\vv_{l})\ot (v_{l}\ot\vv_{l})] &= 0\\
  &= 2\chi(0\notin \inter(l,l)) \\
 &= \Delta(a_lb_l;a_{l}b_{l}). 
 \end{aligned}
\end{equation}
Plugging this into \eqref{eq:H1} gives, for all $k'\neq l$,
\begin{equation}\label{eq:scdcas}
\begin{aligned}
 H[(v_{k'}\ot\vv_{k'})\ot (v_{l}\ot\vv_{l})] &= 1\\
 & = \chi(0\in \inter(l,k'))+\chi(0\in \inter(k',l)) \\
 &= \Delta(a_{l}b_{l};a_{k'}b_{k'}).
 \end{aligned}
\end{equation}
\subsection{The case $k'=l\neq k=l'$}
We now build a path from $(v_{l}\ot\vv_{l})\ot (v_{k}\ot\vv_{k})$ to $(v_{l}\ot\vv_{k})\ot (v_{k}\ot\vv_{l})$.
By \eqref{eq:pp}, we know that 
 $\varepsilon_i(v_{k}\ot\vv_{k})=\chi(i=k)$ and $\varepsilon_i(v_{k}\ot\vv_{l})=0$ if $i\notin\{l+1,k\}$. Since $k\neq l$, we have for all $i\in \inter(k,l)$ that $(v_l\ot\vv_i)\neq (v_l\ot \vv_{l+1})$, and then 
 $(v_l\ot\vv_i)\xrightarrow[]{i} (v_l\ot \vv_{i-1})$. We obtain the path
\begin{align*}
(v_{l}\ot\vv_{l})\ot (v_{k}\ot\vv_{k})\xrightarrow[]{\,\,\,k\,\,}\underbrace{(v_{l}\ot\vv_{l})\ot (v_{k}\ot\vv_{k-1})\xrightarrow[]{k-1}\cdots \xrightarrow[]{l+1}}_{\text{empty if }l+1=k}(v_{l}\ot\vv_{l})\ot (v_{k}\ot\vv_{l})\\
\Big\downarrow \scriptstyle{l}\qquad\qquad \qquad\\
(v_{l}\ot\vv_{k})\ot (v_{k}\ot\vv_{l})\xleftarrow[]{\,\,k+1\,\,}(v_l\ot\vv_{k+1})\ot (v_k\ot\vv_{l})\underbrace{\xleftarrow[]{k+2}\cdots \xleftarrow[]{l-1}(v_{l}\ot\vv_{l-1})\ot (v_{k}\ot\vv_{l})}_{\text{empty if }l=k+1}\\
\end{align*}
In the upper part of the path, we move forward (i.e. with some $\tilde f_i$) by modifying the right side of the tensor product with arrows in $\inter(l,k)$ appearing once. Then, in the lower part of the path, we move forward by modifying the left side of the tensor product with arrows in $\inter(k,l)$ appearing once. Using that $k\neq l$, the energy function satisfies:
$$
\begin{array}{lll}
 H[(v_{l}\ot\vv_{k})\ot (v_{k}\ot\vv_{l})] &=H[(v_{l}\ot\vv_{l})\ot (v_{k}\ot\vv_{k})]+\chi(0\in \inter(l,k))-\chi(0\in \inter(k,l)) & \text{ by \eqref{eq:enfct}}\\
 &=1+2\chi(0\in \inter(l,k))-1 & \text{ by \eqref{eq:scdcas}}\\
 &= \Delta(a_{l}b_{k};a_{k}b_{l})&  \text{ by \eqref{eq:finfin}}.
\end{array}
$$
\subsection{The case $k'=l$ and $k\neq l'$}
The vertices $(v_{l'}\ot\vv_{l})\ot (v_{l}\ot\vv_{k})$ and  $(v_{k}\ot\vv_{l})\ot (v_{l}\ot\vv_{l'})$ are symmetric.

Since $k\neq l'$, we have that $\inter(k,l)\neq \inter(l',l)$. By symmetry, we can assume that $\inter(l',l)\not\subset \inter(k,l)\subset \inter(l',l),$
so that $l'+1\notin \inter(k,l)$. In that case, we necessarily have $k\neq l$. Then,
$\varphi_l (v_{l'}\ot\vv_{l})= 1 = \varepsilon_l(v_{l}\ot\vv_{l})$ and $\varphi_i (v_{l'}\ot\vv_{l})=0$ for all $i\in \inter(k,l)\setminus \{l\}$. Thus we have the path

\begin{align*}
\underbrace{(v_{l}\ot\vv_{l})\ot (v_{l}\ot\vv_{l})\xleftarrow[]{\,\,\,l\,\,}(v_{l-1}\ot\vv_{l})\ot (v_{l}\ot\vv_{l})\xleftarrow[]{l-1}\cdots \xleftarrow[]{l'+1}}_{\text{empty if }l=l'}
(v_{l'}\ot\vv_{l})\ot (v_{l}\ot\vv_{l})\\
\Big\downarrow {\scriptstyle l}\qquad\qquad \qquad\\
(v_{l'}\ot\vv_{l})\ot (v_{l}\ot\vv_{k})\xleftarrow[]{\,\,\,k+1\,\,\,}\cdots\xleftarrow[]{\,\,\,l+1\,\,\,}(v_{l'}\ot\vv_{l})\ot (v_{l}\ot\vv_{l-1})
\end{align*}
and the energy function is given by
$$
\begin{array}{lll}
 H[(v_{l'}\ot\vv_{l})\ot (v_{l}\ot\vv_{k})] &=\chi(l'\neq l)\chi(0\in \inter(l',l))+\chi(0\in \inter(k,l)) & \text{ by \eqref{eq:enfct}}\\
&=\chi(0\notin \inter(l,l'))+\chi(0\notin \inter(l,k)) & \text{ by Lemma \ref{lem:diffint}}\\
 &= \Delta(a_{k}b_{l};a_{l}b_{l'})  & \text{ by \eqref{eq:finfin}}.
\end{array}
$$
This was the last case where $k'=l$. Also, we have already studied a special case where $k'\neq l$, which was the case $l'=k'\neq l=k$. We now study the other cases where $k'\neq l$. 

\subsection{The case $k'\neq k=l=l'$ (Symmetric case: $l\neq k=k'=l'$)}
Since $l\notin \inter(l,k')$, we have the path
\begin{align*}
(v_{l+1}\ot\vv_{l+1})\ot (v_{l}\ot\vv_{l})\xleftarrow[]{\,l+1\,\,}\underbrace{(v_{l}\ot\vv_{l+1})\ot (v_{l}\ot\vv_{l})\xleftarrow[]{l+2}\cdots \xleftarrow[]{\,\,\,k'\,\,\,}}_{\text{empty if }k'=l+1}
(v_{l}\ot\vv_{k'})\ot (v_{l}\ot\vv_{l})\,.
\end{align*}
Thus the energy function satisfies
$$
\begin{array}{lll}
 H[(v_{l}\ot\vv_{k'})\ot (v_{l}\ot\vv_{l})] &=1+\chi(0\in \inter(l,k')) & \text{ by \eqref{eq:enfct} and \eqref{eq:scdcas}}\\
&=\chi(0\in \inter(l,l))+\chi(0\in \inter(l,k')) & \text{ by Lemma \ref{lem:diffint}}\\
 &= \Delta(a_{l}b_{l};a_{k'}b_{l}) & \text{ by \eqref{eq:finfin}}.
\end{array}
$$
\subsection{The case $l'\neq k'=k\neq l$ (Symmetric case: $k\neq l=l' \neq k'$)}
We first assume that $l'+1\notin \inter(k',l)$. Since $l'\neq k'$, it means that 
\[\inter(l',k')\sqcup \inter(k',l) = \inter(l',l).\]
Since $l'+1$ and $k'$ do not belong to $\inter(k',l),$ we have by \eqref{eq:pp} that $\varphi_i(v_{l'}\ot\vv_{k'})=0$ for all $i\in \inter(k',l).$
This gives the path 
\begin{align*}
(v_{l'+1}\ot\vv_{l'+1})\ot (v_{k'+1}\ot\vv_{k'+1})\xleftarrow[]{\,l'+1\,\,}\underbrace{(v_{l'}\ot\vv_{l'+1})\ot (v_{k'+1}\ot\vv_{k'+1})\xleftarrow[]{l'+2}\cdots \xleftarrow[]{\,\,\,k'\,\,\,}}_{\text{empty if }k'=l'+1}
(v_{l'}\ot\vv_{k'})\ot (v_{k'+1}\ot\vv_{k'+1})\\
\Big\downarrow {\scriptstyle k'+1}\qquad \qquad\\
(v_{l'}\ot\vv_{k'})\ot (v_{l}\ot\vv_{k'})\underbrace{\xleftarrow[]{\,\,\,l\,\,\,}\cdots\xleftarrow[]{\,\,\,k'+2\,\,\,}(v_{l'}\ot\vv_{k'})\ot (v_{k'+1}\ot\vv_{k'})}_{\text{empty if }k'+1=l}\,\cdot
\end{align*}
We deduce the following formula for the energy function:
$$
\begin{array}{lll}
 H[(v_{l'}\ot\vv_{k'})\ot (v_{l}\ot\vv_{k'})] &=1+ \chi(0\in \inter(l',k'))+\chi(0\in \inter(k',l)) & \text{ by \eqref{eq:enfct} and \eqref{eq:scdcas}}\\
&=\chi(0\in \inter(k',k'))+\chi(0\in \inter(l',l)) & \text{ by Lemma \ref{lem:diffint}}\\
 &= \Delta(a_{k'}b_{l};a_{k'}b_{l'})& \text{ by \eqref{eq:finfin}}.
\end{array}
$$
\m Let us now assume that $l'+1\in \inter(k',l)$. Since $\inter(k',l)\neq \emptyset $ and $l'\neq k'$, we necessarily have that
$k'+1\neq l$ and $\inter(k',l')\subset \inter(k',l-1)$, so that $l'\neq l$. Note also that, by \eqref{eq:pp},  
$$\varphi_{k'}(v_{l'}\ot\vv_{k'-1})=0=\varepsilon_{k'}(v_{k'-1}\ot\vv_{k'}),$$ since $k'\neq l'+1$,
and $\varphi_{i}(v_{l'}\ot\vv_{k'})=0$ for all $i\in \inter(l,k')\setminus\{k'\}$.
We then have the path
\begin{align*}
(v_{k'}\ot\vv_{k'-1})\ot (v_{k'}\ot\vv_{k'-1})\underbrace{\xleftarrow[]{\,k'\,\,}}_{\bullet}\underbrace{(v_{k'}\ot\vv_{k'-1})\ot (v_{k'}\ot\vv_{k'})\xrightarrow[]{k'+1}\cdots \xrightarrow[]{\,\,\,l'\,\,\,}}_{\text{nonempty since }k'\neq l'+1}
(v_{l'}\ot\vv_{k'-1})\ot (v_{k'}\ot\vv_{k'})\\
\underbrace{\Big\uparrow {\scriptstyle k'}}_{\star}
\qquad \qquad\\
\underbrace{(v_{l'}\ot\vv_{k'})\ot (v_{l}\ot\vv_{k'})\xrightarrow[]{\,\,\,\,l+1\,\,\,\,}\cdots\xrightarrow[]{\,\,k'-1\,\,}}_{\bullet}(v_{l'}\ot\vv_{k'})\ot (v_{k'-1}\ot\vv_{k'})\underbrace{\xrightarrow[]{\,\,\,k'\,\,\,}}_{\star}(v_{l'}\ot\vv_{k'-1})\ot (v_{k'-1}\ot\vv_{k'})\,\cdot
\end{align*}
By the previous case ($l'\neq k'=k\neq l$), we obtain the energy function
\begin{equation}
\label{eq:Hdifferent}
\begin{aligned}
H[(v_{k'}\ot\vv_{k'-1})\ot (v_{k'}\ot\vv_{k'-1})] &= \chi(0\in \inter(k',k'))+\chi(0\in \inter(k'-1,k'-1))\\
&=2\chi(0\in \inter(k',k')).
\end{aligned}
\end{equation}
In the computation of $H$, by \eqref{eq:enfct}, the moves marked by $\star$  cancel each other, since it is the same arrow that operates backward (i.e. by some $\tilde{e}_i$) consecutively on the right and on the left side of the tensor product. 
Besides, the moves marked by $\bullet$ give $\inter(l,k')$ and operate backward on the right side of the tensor product. As a consequence,
$$
\begin{array}{lll}
 H[(v_{l'}\ot\vv_{k'})\ot (v_{l}\ot\vv_{k'})] &=H[(v_{k'}\ot\vv_{k'-1})\ot (v_{k'}\ot\vv_{k'-1})]-\chi(0\in \inter(k',l'))- \chi(0\in \inter(l,k')) &\text{ by \eqref{eq:enfct}}\\
 &=2\chi(0\in \inter(k',k'))-\chi(0\in \inter(k',l'))- \chi(0\in \inter(l,k')) &\text{ by \eqref{eq:Hdifferent}}\\
 &= \chi(0\in \inter(k',k'))+ \chi(0\in \inter(k',l))-\chi(0\in \inter(k',l'))\\
 &=\chi(0\in \inter(k',k'))+\chi(0\in \inter(l',l)) &\text{ by Lemma \ref{lem:diffint}}\\
 &= \Delta(a_{k'}b_{l};a_{k'}b_{l'}) & \text{ by \eqref{eq:finfin}}.
\end{array}
$$

\subsection{The case $k\neq k'$, $k'\neq l$ and $l\neq l'$}
\subsubsection{The sub-case $k+1,k'\notin \inter(l,l')$ (Symmetric case : $l'+1 ,l\notin \inter(k',k)$)}
We have $l'+1,k'\notin \inter(l,l')$, so that $\varphi_i(v_{l'}\ot\vv_{k'})=0$ for all $i\in \inter(l,l')$. Besides, $k+1\notin \inter(l,l')$, so that 
$\tilde e_i(v_{i}\ot\vv_{k})=(v_{i-1}\ot\vv_{k})$. We obtain the path
\begin{align*}
(v_{l'}\ot\vv_{k'})\ot (v_{l'}\ot\vv_{k})\xleftarrow[]{\,\,\,l'\,\,}\cdots \xleftarrow[]{\,\,\,l+1\,\,\,}(v_{l'}\ot\vv_{k'})\ot (v_{l}\ot\vv_{k}).
\end{align*}
By Case 7.4 and the symmetric of Case 7.5, we have
\begin{equation}
\label{eq:regarde}
H[(v_{l'}\ot\vv_{k'})\ot (v_{l'}\ot\vv_{k})]=\chi(0\in \inter(k,k'))+\chi(0\in \inter(l',l')),
\end{equation}
and the energy function becomes
$$
\begin{array}{lll}
 H[(v_{l'}\ot\vv_{k'})\ot (v_{l}\ot\vv_{k})] &= H[(v_{l'}\ot\vv_{k'})\ot (v_{l'}\ot\vv_{k})]-\chi(0\in \inter(l,l')) &\text{ by \eqref{eq:enfct}}\\
 &=\chi(0\in \inter(k,k'))+\chi(0\in \inter(l',l'))-\chi(0\in \inter(l,l')) &\text{ by \eqref{eq:regarde}}\\
 &=\chi(0\in \inter(k,k'))+\chi(0\in \inter(l',l))  &\text{ by Lemma \ref{lem:diffint}}\\
 &= \Delta(a_{k}b_{l};a_{k'}b_{l'})& \text{ by \eqref{eq:finfin}}.
\end{array}
$$

\subsubsection{The sub-case $k+1\in \inter(l,l')$ and $k'\notin \inter(l,l')$ (Symmetric case: $l'+1 \in \inter(k',k)$ and $l\notin \inter(k',k)$)}
This case is very similar to the previous one. 
%We still have $\varphi_i(v_{l'}\ot\vv_{k'})=0$ for all $i\in \inter(l,l')$, except that we have to pass the vertex $(v_{k+1}\ot\vv_{k})$. We proceed as follows:
We use the following path:
\begin{align*}
\underbrace{(v_{l'}\ot\vv_{k'})\ot (v_{l'}\ot\vv_{k})\xleftarrow[]{\,\,\,l'\,\,}\cdots \xleftarrow[]{\,\,\,k+2\,\,\,}(v_{l'}\ot\vv_{k'})\ot (v_{k+1}\ot\vv_{k})}_{\star}
\underbrace{\xrightarrow[]{\,\,\,k\,\,\,}}_{\bullet}(v_{l'}\ot\vv_{k'})\ot (v_{k+1}\ot\vv_{k-1})\\
\underbrace{\Big\uparrow {\scriptstyle k+1}}_{\star}\\
\underbrace{(v_{l'}\ot\vv_{k'})\ot (v_{l'}\ot\vv_{k})\xrightarrow[]{\,\,\,l+1\,\,}\cdots \xrightarrow[]{\,\,\,k\,\,\,}(v_{l'}\ot\vv_{k'})\ot (v_{k}\ot\vv_{k})}_{\star}\underbrace{\xrightarrow[]{\,\,\,k\,\,\,}}_{\bullet}(v_{l'}\ot\vv_{k'})\ot (v_{k}\ot\vv_{k-1})
\end{align*}
Note that the moves marked by $\bullet$ cancel each other, and the moves marked by $\star$ give $\inter(l,l')$, so that the calculation is the same as in the previous case.
\subsubsection{The sub-case $k+1\notin \inter(l,l')$ and $k'\in \inter(l,l')$ (Symmetric case: $l'+1 \notin \inter(k',k)$ and $l\in \inter(k',k)$)}
We have $l,k+1\notin \inter(l,l')$, so that $\varepsilon_i(v_{l}\ot\vv_{k})=0$ for all $i\in \inter(l,l')$. Note that $k'+1\in \inter(l,l')$, since $k'\in \inter(l,l')$ and $k'\neq l'$. This gives the path
\begin{align*}
\underbrace{(v_{l}\ot\vv_{k'})\ot (v_{l}\ot\vv_{k})\xrightarrow[]{\,\,\,l+1\,\,}\cdots \xrightarrow[]{\,\,\,k'\,\,\,}(v_{k'}\ot\vv_{k'})\ot (v_{l}\ot\vv_{k})}_{\star}
\underbrace{\xrightarrow[]{\,\,\,k'\,\,\,}}_{\bullet}(v_{k'}\ot\vv_{k'-1})\ot (v_{l}\ot\vv_{k})\\
\underbrace{\Big\downarrow {\scriptstyle k'+1}}_{\star}\\
\underbrace{(v_{l'}\ot\vv_{k'})\ot (v_{l}\ot\vv_{k})\xleftarrow[]{\,\,\,l'\,\,}\cdots \xleftarrow[]{\,\,\,k'+2\,\,\,}(v_{k'+1}\ot\vv_{k'})\ot (v_{l}\ot\vv_{k})}_{\star}\underbrace{\xrightarrow[]{\,\,\,k'\,\,\,}}_{\bullet}(v_{k'+1}\ot\vv_{k'-1})\ot (v_{l}\ot\vv_{k})
\end{align*}
As before, the moves marked by $\bullet$ cancel each other, and the moves $\star$ give $\inter(l,l')$. We move with the $\tilde f_i$'s by changing the left side of the tensor product, and we get
$$
\begin{array}{lll}
 H[(v_{l'}\ot\vv_{k'})\ot (v_{l}\ot\vv_{k})] &= H[(v_{l}\ot\vv_{k'})\ot (v_{l}\ot\vv_{k})]-\chi(0\in \inter(l,l')) &\text{ by \eqref{eq:enfct}}\\
 &=\chi(0\in \inter(k,k'))+\chi(0\in \inter(l,l))-\chi(0\in \inter(l,l')) &\text{ by \eqref{eq:regarde}}\\
 &=\chi(0\in \inter(k,k'))+\chi(0\in \inter(l',l)) &\text{ by Lemma \ref{lem:diffint}}\\
 &= \Delta(a_{k}b_{l};a_{k'}b_{l'}) & \text{ by \eqref{eq:finfin}}.
\end{array}
$$
\subsubsection{The sub-case $k+1,k'\in \inter(l,l')$ and $l'+1,l\in \inter(k',k)$}
Note that this case overlaps with the case $k'=l'\neq k=l$ that we already checked in the first part.
Omitting that case, we can assume by symmetry that $k\neq l$.
We obtain the path
\begin{align*}
\underbrace{(v_{l'}\ot\vv_{l'})\ot (v_{k}\ot\vv_{k})\xrightarrow[]{\,\,\,l'\,\,}\cdots \xrightarrow[]{\,k'+1\,}}_{\text{empty if }k'=l'}(v_{l'}\ot\vv_{k'})\ot (v_{k}\ot\vv_{k})
\xleftarrow[]{\,\,\,k\,\,\,}\cdots\xleftarrow[]{\,\,\,l+1\,\,\,}(v_{l'}\ot\vv_{k'})\ot (v_{l}\ot\vv_{k})\,\cdot
\end{align*}
Since $k\neq l$, the fact that $l\in \inter(k',k)$  implies that $ \inter(k',k)= \inter(k',l)\sqcup \inter(l,k)$, and the fact that 
$k+1\in \inter(l,l')$  implies that $\inter(l,l')= \inter(l,k)\sqcup \inter(k,l'),$ so that $k',l'+1\notin \inter(l,k)$.
Also, if $k'\neq l'$, then $l'+1\in int (k',k)$ implies that $\inter(k',k)=\inter(k',l')\sqcup \inter(l',k)$, so that $k\notin \inter(k',l')$.
Since $l\neq l'$ and $k'\neq l$, the fact that $k'\in \inter(l,l')$ implies that 
\[\inter(l',k')=\inter(l',l)\sqcup \inter(l,k'),\]
and 
the fact that $l\in \inter(k',k)$ and $l\neq k$ implies that 
\[\inter(l,k')=\inter(l,k)\sqcup \inter(k,k').\]
Thus the computation of $H$ gives
$$
\begin{array}{lll}
 H[(v_{l'}\ot\vv_{k'})\ot (v_{l}\ot\vv_{k})] &= 1-\chi(k'\neq l')\chi(0\in \inter(k',l'))-\chi(0\in \inter(l,k))  & \text{ by \eqref{eq:enfct} and \eqref{eq:scdcas}}\\
 &=1-\chi(0\notin \inter(l',k'))-\chi(0\in \inter(l,k))  & \text{ by Lemma \ref{lem:diffint}}\\
 &=\chi(0\in \inter(l',k'))-\chi(0\in \inter(l,k))\\
  &=\chi(0\in \inter(l',l))+\chi(0\in \inter(l,k'))-\chi(0\in \inter(l,k))\\
   &=\chi(0\in \inter(l',l))+\chi(0\in \inter(k,k'))\\
  &=\Delta(a_kb_l;a_{k'}b_{l'})& \text{ by \eqref{eq:finfin}}.
\end{array}
$$

\medskip
We have checked all the possible choices of $k,l,k',l'$. Our proof of \Thm{theorem:perfcrys} is thus complete.

%\begin{thebibliography}{AAAA}
%
%\bibitem[{DK}]{DK19} J. Dousse and I. Konan, generalisations of Capparelli's and Primc's identities I: coloured Frobenius partitions and combinatorial proofs.
%
%\bibitem[HK]{HK} J.~Hong and S.-J. Kang, {\it Introduction to Quantum Groups and Crystal Bases},
%Grad. Studies in Math. \textbf{42}, Amer. Math. Soc., Providence, 2002.
%
%\bibitem[K]{K}V.~Kac, {\it Infinite Dimensional Lie Algebras},  3rd ed.  Cambridge Univ. Press,
%Cambridge, 1990.
%
%\bibitem[KKM]{KKM} S.-J.~Kang, M.~Kashiwara, and K.C.~Misra,   Crystal bases of Verma modules for quantum affine
%Lie algebras,  Compositio Math. \textbf{92} (1994), 299-325.
%
%\bibitem[(KMN)$^2$a]{(KMN)$^2$a} S.-J.~Kang, M.~Kashiwara, K.C.~Misra,
%T.~Miwa, T.~Nakashima, and A.~Nakayashiki,  \ Affine crystals and vertex models,
% Infinite Analysis, Part A, B (Kyoto, 1991),  449--484, Adv. Ser. Math. Phys. \textbf{16},  World Sci. Publishing, River Edge, NJ, 1992.
%
%
%\bibitem[(KMN)$^2$b]{(KMN)$^2$b} S.-J.~Kang, M.~Kashiwara, K.C.~Misra,
%T.~Miwa, T.~Nakashima, and A.~Nakayashiki,  Perfect crystals of quantum affine
%Lie algebras, Duke Math. J. \textbf{68} (1992), 499-607.
%
%
%\bibitem[{KMPY}]{KMPY}
%M.~Kashiwara, T.~Miwa,  J.-U.H.~Petersen,  and C.M.~Yung,  Perfect crystals and $q$-deformed Fock spaces.  Selecta Math. (N.S.)  {\textbf 2}  (1996),  no. 3, 415--499.
%
%
%\bibitem[Y]{Y}S.~Yamane, Perfect crystals of $U_q($G$_2^{(1)})$, J. Algebra
%\textbf{210} (1998), 440-486.
%
%\bigskip
%
%\end{thebibliography}

\section*{Acknowledgements}
The research of the first author was supported by the project IMPULSION of IdexLyon. Part of this research was conducted while the second author was visiting Lyon, funded by the same project.

The authors would like to thank Leonard Hardiman, Travis Scrimshaw, and Ole Warnaar for their helpful comments on an earlier version of this paper and for suggesting important references.

\bibliographystyle{alpha}
\bibliography{biblio}

\end{document}